\documentclass[12pt,amscd]{amsart}
\usepackage[all]{xy}
\usepackage{graphicx}
\usepackage{amsmath,amsxtra,amssymb,latexsym, amscd,amsthm}
\usepackage{indentfirst}
\usepackage[mathscr]{eucal}    

\newtheorem{thm}{Theorem}[section]
\newtheorem{cor}[thm]{Corollary}
\newtheorem{lem}[thm]{Lemma}
\newtheorem{prop}[thm]{Proposition}
\theoremstyle{definition}
\newtheorem{defn}[thm]{Definition}
\newtheorem{exm}[thm]{Example}
\newtheorem{rem}[thm]{Remark}
\newtheorem{quest}[thm]{Question}

\numberwithin{equation}{section}

\DeclareMathOperator{\depth}{depth}
\DeclareMathOperator{\reg}{reg}
\DeclareMathOperator{\inte}{int}
\DeclareMathOperator{\supp}{supp}

\DeclareMathOperator{\regst}{reg-stab}
\DeclareMathOperator{\conv}{conv}
\DeclareMathOperator{\cone}{cone}
\DeclareMathOperator{\nmod}{mod}
\DeclareMathOperator{\rank}{rank}
\def\regb{\overline{\regst}}

\def\Nset{\mathbb {N}}
\def\Zset{\mathbb {Z}}
\def\Qset{\mathbb {Q}}
\def\Rset{\mathbb {R}}
\def\Ga {\supp^{-}(\albf)}
\def\Da {\Delta_{\albf}}

\def\Inba {\overline{I^n}}
\def\Imba {\overline{I^m}}

\def\albf {{\boldsymbol{\alpha}}}
\def\bebf {{\boldsymbol{\beta}}}

\def\gabf {{\boldsymbol{\gamma}}}

\def\zerobf {\mathbf 0}
\def\abf {\mathbf a}
\def\At {\widetilde{A}}
\def\bT{\widetilde{\mathbf b}}
\def\bbf {\mathbf b}
\def\cbf {\mathbf c}
\def\cT{\widetilde{\mathbf c}}
\def\ct{\tilde{c}}
\def\dbf {\mathbf d}
\def\dT{\widetilde{\mathbf d}}
\def\ebf {\mathbf e}
\def\Ibf {\mathbf I}
\def\xbf {\mathbf x}
\def\Xbf {\mathbf X}
\def\ybf {\mathbf y}
\def\zbf {\mathbf z}
\def\ubf {\mathbf u}
\def\vbf {\mathbf v}
\def\wbf {\mathbf w}

\def\mfr {\mathfrak m}

\def\pfr {\mathfrak p}

\def\Fcal{\mathcal F}
\def\Pcal{\mathcal P}
\def\Ptil{\tilde{\mathcal P}}
\def\Qcal{\mathcal Q}
\def\Qtil{\tilde{\mathcal Q}}
\def\Ical{\mathcal I}
\def\Jcal{\mathcal J}
\def\Ecal{\mathcal E}

\def\stil {\tilde{s}}

\begin{document}

\title[Integer  Programming and  Castelnuovo-Mumford regularity] {Asymptotic behavior of    Integer Programming and the stability of the Castelnuovo-Mumford regularity}
\author{Le Tuan Hoa }
\address{Institute of Mathematics, VAST, 18 Hoang Quoc Viet, 10307 Hanoi, Viet Nam}
\email{lthoa@math.ac.vn}
\date{}
\dedicatory{Dedicated to Professor Ngo Viet Trung on the occasion of his 65th birthday.}
\commby{}
\begin{abstract} The paper  provides a connection between  Commutative Algebra and Integer Programming and contains two parts. The first one is devoted to the asymptotic behavior of integer programs with a fixed cost linear functional and the constraint sets consisting of a finite system of linear equations or inequalities with integer  coefficients depending linearly on $n$.  An integer $N_*$ is determined such that the optima of  these   integer  programs  are a quasi-linear function of $n$ for all  $n\ge N_*$. Using results in the first part, one can bound in the second part the indices of stability  of the Castelnuovo-Mumford regularities of  integral closures of powers of a  monomial ideal and  that of  symbolic powers of a square-free monomial ideal. 
\end{abstract}

\maketitle
\section*{Introduction}

This paper provides a case when a problem in Integer Programming is raised from  Commutative Algebra and its solution leads to solving the original problem in the latter field. 

Let $I$ be a proper homogeneous ideal of a polynomial ring $R= K[X_1,...,X_r]$ over a field $K$. The Castelnuovo-Mumford regularity $\reg(R/I)$ (see (\ref{Ereg2}) for the definition) is one of the most important invariants of $I$. This notion was introduced by D. Mumford for sheaves in 1966, and then extended to graded modules by D. Eisenbud and S. Goto in 1984, see Chapter 4 of  \cite{Ei}.  Let $\bar{I}$ denote the integral closure of $I$.  Very few are known about $\reg(R/I^n)$ and  $\reg(R/\Inba)$. However, if $n\gg 0$ then Cutkosky-Herzog-Trung \cite{CHT} and independently Kodiyalam \cite{Ko} proved that these invariants are linear functions of $n$. Unfortunately the proofs in both papers do not give any hint to when these invariants become linear. In order to address this problem, indices of stability are introduced, see Definition \ref{CMst}. So,  $\regst (I)$ (resp., $\regb(I)$) is the smallest number such that $\reg(R/I^n)$ (resp., $\reg(R/\Inba)$) is a linear function for all $n\ge \regst(I)$ (resp., $n\ge \regb(I)$).

 It is of great interest to bound $\regst (I)$  and $\regb(I)$. However these problems seem to be very difficult and  are currently solved for only few cases. If $I$ is an $\mfr$-primary ideal, where $\mfr = (X_1,...,X_r)$, then some bounds on $\regst(I)$ are established in \cite{Ber, C, EU} in terms of other invariants, which are not easy to compute. Under the additional assumption that $I$ is a monomial ideal, an explicit bound is given in \cite[Theorem 3.1]{Ber}. If $I$ is not an $\mfr$-primary ideal, then in a series of papers it is shown that for some very special monomial ideals (often generated by square-free monomials of degree two), $\regst (I)$ is a rather small number, see \cite{ABan, ABS, Ban, BHT, HaT, JNS, JS}. 

Our purpose is to provide an explicit bound on  $\regb(I)$ for all monomial ideals,  in terms of the maximal generating degree of $I$ and the number $r$ of variables. In order to do that, together with studying $\reg(R/\Inba)$ we also study the so-called $a_i$-invariants $a_i(R/\Inba)$, which can be regarded as partial Castelnuovo-Mumford regularities, see (\ref{Ereg1}). Using a technique of computing local cohomology modules $H^i_{\mfr}(R/I)$ of a monomial ideal given in \cite{Ta} and developed further in a series of papers \cite{GH, HT1, HT3, MT, Tr1}, one can translate the problem of computing $H^i_{\mfr}(R/\Inba)$ into studying the sets of integer points in some rational polyhedra. 
Now, dealing with partial Castelnuovo-Mumford regularities, for each $n$, the computation of 
 $a_i(R/\Inba)$ can be formulated as  a finite set of integer programs, see Theorem \ref{CM7} and Corollary \ref{TakayB}.  This is a crucial point in our approach, which allows to connect two branches of Mathematics. 

Now we get a parametric family (depending on $n$) of  integer programs, and one way to study the asymptotic behavior of the $a_i$-invariants is to study  the behavior of the maximum $M_n$, as a  function of $n$, of  the integer program   $(IQ_n)$: 
 $$\begin{array}{ll} (IQ_n) \ \ & \max\{d_1x_1 + \cdots + d_rx_r|\ \xbf \in \Qcal_n \cap \Nset^r\},\\
 \text{where} & \Qcal_n =  \{\xbf \in \Rset^r |\ \sum a_{ij} x_j \le nb_i + c_i  \ \ \forall i\le s; \\ & \text{and} \ x_j \ge 0 \ \forall j =1,...,r\},
 \end{array}$$
and $a_{ij}, d_i, b_i, c_i\in \Zset$.  

This idea was applied in a recent paper \cite{HaT}, but the constraint set there is much more simple. In our case, we have to deal with a rather general situation.  If all $c_i =0$, then for each $n$, the set of feasible solutions to the  corresponding integer program is exactly the set of integer points in  the convex polyhedron:
$$\Pcal_n =   \{\xbf \in \Rset^r |\ \sum a_{ij} x_j \le nb_i  \ \ \forall i\le s; \ \text{and} \ x_j \ge 0 \ \forall j =1,...,r\}.$$
Of course, we can restrict to the case $\Pcal := \Pcal_1$ is not empty. One can consider $\Qcal_n$ as a relaxation of  $\Pcal_n = n\Pcal$.  If all vertices of $\Pcal$ are integral, then so is $\Pcal_n$, and the maximum over integer points of $\Pcal_n$ is a linear function of $n$ for all $n$, since the maximum of the corresponding linear program is attained at a vertex of $\Pcal_n = n\Pcal$ (this is the case in \cite{HaT}). However, if not all vertices of $\Pcal$ are integral or  if $c_i \neq 0$ for some $i$, then even the feasibility of the integer program  $(IQ_n)$ is  unclear; that means it is unclear if the polyhedron $\Qcal_n $ is empty or not, and if it has an integer point. Note that the study of integer points in the family of the polyhedra $\Pcal_n$ was initiated by Ehrhart  in \cite{E1, E2}, which has lead to many applications and stimulated a  lot of research until the present time. There are also interesting relationships between rational points in integral polyhedra and Commutative Algebra, see e.g., Stanley's book \cite{St}. 

The study of asymptotic integer programs was initiated in the work by Gomory \cite{Go}, where a  much more general situation is considered: the right hand side of constraints can be any vector in some cone. He 
showed certain asymptotic periodicity of the optimal solution. His study was then extended by Wolsey 
in \cite{Wol}.  From their result, Theorem \ref{GoW1}, one can quickly derive that $M_n$ is an 
eventually quasi-linear function, see Proposition \ref{GoW2}.  Note that an explicit formulation (for a 
much more general situation) of this fact can be also found  in recent papers \cite{Wo, Sh}.  

However, for our application, the main concern here is to  determine a number $N_*$ in terms of data $A = (a_{ij}), \bbf = (b_i), \cbf = (c_i), \dbf = (d_i)$, such that $M_n$ becomes a quasi-linear function of $n$ for all $n \ge N_*$.  In order to do this, we have to find a new and independent proof of the fact that $M_n$  is an eventually quasi-linear function, see Theorem \ref{I9}. In the proof, we need some facts on the finite generation of a semigroup ring (see Lemma \ref{I3}) and Schur's bound on the Frobenius  number of a finite sequence of positive integers.

Back to the study of the $a_i$-invariants,  we further need an auxiliary result to guarantee that we can apply Theorem \ref{I9} of the first part,  see Lemma \ref{CM8}. From that we finally can show that  $a_i(R/\Inba)$ is a quasi-linear function of the same slope for all $n\ge N_\dagger $, where $N_\dagger $ is explicitly defined in terms of the maximal generating degree $d(I)$ of $I$ and $r$, see Theorem \ref{CM10}. A bound on $\regb(I)$ can be now quickly derived, see Theorem \ref{CM11}.

The technique of bounding $\regb(I)$ can be applied to studying  the asymptotic behavior of the Castelnuovo-Mumford regularity $\reg(R/I^{(n)})$ and $a_i(R/I^{(n)})$ of the so-called symbolic powers $I^{(n)}$ of $I$, where $I$ is an arbitrary square-free monomial ideal.  In this case we are  able to show that $a_i(R/I^{(n)})$ and $\reg(R/I^{(n)})$ are quasi-linear functions of $n$ for all $n\ge 2r^{2 + 3r/2}$, see Theorem \ref{Sym2}.

Now we briefly describe the content of the paper. In Section \ref{Opt} we give some auxiliary results. Here a condition for the feasibility of the set of solutions to  the corresponding linear programs  is presented. The asymptotic behavior of Integer Programming is studied in Section \ref{IOpt}. Firstly, a condition on $n$ is given to guarantee the existence of integer   solutions, see Lemma \ref{I1}. Secondly, we show how to use a result by Gomory and Wolsey to derive the property of $M_n$ being a quasi-linear function for $n\gg 0$. Then, we give some properties of the maxima $m_n$ and $M_n$ of the integer programs over $n\Pcal$ and its relaxation $\Qcal_n$, see Lemma \ref{I4}, Lemma \ref{I6} and Remark \ref{I9a}. Finally, we can  give a proof of the main result of this section, Theorem \ref{I9}.  The study of the asymptotic behavior of  $a_i(R/\Inba)$ and  $\reg(R/\Inba)$ as well as of  $a_i(R/I^{(n)})$ and $\reg(R/I^{(n)})$ is carried out in Section \ref{CM}. After describing the Newton polyhedron of $I$ (Lemma \ref{NPH}) and simplicial complexes associated to $\Inba$ (Lemma \ref{FNPn}) we give a proof of Theorem \ref{CM7} connecting the two branches in Mathematics. Using this theorem and Theorem \ref{I9} we can formulate and prove the main results of the papers, Theorems \ref{CM10}, \ref{CM11} and \ref{Sym2}.

\section{Asymptotic behavior of   Linear Programming} \label{Opt}

In this section we present some preliminary results which give the feasibility of the set of solutions to the corresponding linear programs  as well as some properties of its asymptotic behavior.  These results will be used in the next sections.

The symbols $\Zset,\ \Nset,\ \Qset$, $\Rset$ and $ \Rset_+$ denote the sets of integers, non-negative integers, rationals, real numbers and non-negative real numbers, respectively. Vectors  usually (but not always) are column vectors.  If $\abf = (a_1,..., a_r)$ and $\abf' = (a_1,...,a_r)$ are row vectors, we write $\abf \le \abf'$ if $a_i \le a'_i$ for all $i=1,...,r$. Similarly for column vectors. For two subsets $U, V \subseteq  \Rset^r$ and $\alpha \in \Rset$, let
$$U + V =\{ u+v|\ u\in U \  \text{and} \ v\in V\}, \ \ \text{and}\ \ \alpha V =\{\alpha v|\ v\in V\}.$$
 We refer the reader to the book \cite{Sch} for unexplained terminology  and notions in Linear and  Integer Programming. Given  a $(s\times r)$-real matrix 
$A = (a_{ij}) \in M_{s,r}(\Rset)$,   vectors $\bbf = (b_1,...,b_s)^T, \ \cbf = (c_1,...,c_s)^T\in \Rset^s$, and a positive integer $n$, let us consider the following polyhedra:
$$\begin{array}{ll} \Pcal  & = \{\xbf \in \Rset^r |\ A \xbf  \le  \bbf; \ \text{and} \ x_j \ge 0 \ \forall j =1,...,r\}, \\
\Pcal_n  & = \{\xbf \in \Rset^r |\ A \xbf  \le  n\bbf; \ \text{and} \ x_j \ge 0 \ \forall j =1,...,r\},\\
\Qcal_n  & = \{\xbf \in \Rset^r |\ A \xbf  \le  n\bbf + \cbf; \ \text{and} \ x_j \ge 0 \ \forall j =1,...,r\}.
\end{array}$$

In \cite{E2},  Ehrhart called the system $A \xbf  \le  n\bbf$ (i.e.,  $\cbf = 0$) homothetic. Otherwise it is called a bordered system.
Note that  $\Pcal_n = n\Pcal$. So $\Pcal \neq \emptyset $ if and only if $\Pcal_n \neq \emptyset $.  We often use the following property
$$\Pcal_n + \Qcal_m \subseteq \Qcal_{n+m},$$ 
for all $n, m \in \Nset$. In particular, $\Qcal_{n+m} \neq \emptyset $, provided that both $\Pcal_n$ and $\Qcal_m$ are not empty. The following result gives a lower bound  such that starting from this number, all $\inte(\Qcal_n) \neq \emptyset $, where $\inte(*)$ denotes the interior of $*$.

\begin{lem}\label{O1} Assume that $\Pcal$ is full-dimensional. Let $\gabf \in \inte(\Pcal)$. Set
\begin{equation} \label{EO11} \varepsilon_\gabf = \min\{b_i - A_i\gabf|\ i = 1,...,s\},
\end{equation}
where $A_i$ denotes  the $i$-th row of $A$, and 
 \begin{equation} \label{EO12} N_\gabf = 1 +\frac{1}{ \varepsilon_\gabf} \max\{0, \ - c_1,..., - c_s\}.
\end{equation}

Then $n\gabf \in \inte(\Qcal_n)$ for all $n\ge N_\gabf$.
\end{lem}

\begin{proof} If $\gabf \in \inte(\Pcal)$, then $n\gabf \in \inte(\Pcal_n)$. When $n$ is small, $n\gabf$ does not necessarily belong to the relaxation $\Qcal_n$ of $\Pcal_n$. However, since the slack $\varepsilon_{n\gabf}^i := nb_i - A_i(n\gabf)>0 $ of $n\gabf$ in the $i$-th constraint increases linearly on $n$,  the point $n\gabf$ also satisfies the $i$-th constraint of $\Qcal_n$ for all $n$ large enough. More precisely,   for all  $n \ge  N_\gabf  > -c_i/\varepsilon_\gabf \ge - c_i\varepsilon_\gabf^i$, we have   $nb_i - A_i(n\gabf) = n \varepsilon_\gabf^i > c_i$ for all $i$, which yields  $n\gabf \in \inte(\Qcal_n)$.
\end{proof}
As pointed out by a referee, one can  obtain a smaller $N_{\gabf}$ using the number 
$$1 +  \max_{1\le i \le s}\{0, \ - c_i/\varepsilon_\gabf^i \}.$$
 However, we prefer to use (\ref{EO12}), which in some cases simplifies our computation.

In the above lemma, the bound $N_{\gabf}$ depends on the existence of  the interior point $\gabf$. In the following lemma, using the decomposition of $\Pcal$ into the sum of a polytope and a polyhedral cone, one can find an $r$-simplex contained in $\Pcal$. Taking the barycenter of this simplex as  $\gabf$,  we can give an explicit bound $N_0$  for $N_\gabf$ in terms of $\Pcal$.
 
\begin{lem}\label{O2}
Assume that $\Pcal$ is full-dimensional and  $\Pcal$ admits the following decomposition:
$$\Pcal = \conv (\albf_1, ..., \albf_p) + \cone(\bebf_1, ..., \bebf_q),$$
for some points $\albf_i \in \Rset^r$ and direction vectors $\bebf_j$. Set $\bebf_0 = \zerobf$,
\begin{equation} \label{EO21} 
\varepsilon_0 = \min_{i,j,k}\{ b_i - A_i(\albf_j + \bebf_k)|\  A_i(\albf_j + \bebf_k) < b_i \},
\end{equation}
and
\begin{equation} \label{EO22} N_0 = 1 +\frac{r+1}{ \varepsilon_0} \max\{0, - c_1,..., - c_s\}.
\end{equation}

Then, for all $n\ge N_0$, $\inte(\Qcal_n) \neq \emptyset $.
\end{lem}

\begin{proof} Under the assumption one can choose $r+1$ points $\gabf_1,...,\gabf_{r+1}$ from the set 
$ \{ \albf_i + \bebf_j|\ i=0,...,p\ \text{and}\ j=0,...,q\}$  such that $\gabf_1,...,\gabf_{r+1}$ are affinely independent.
Set
\begin{equation} \label{EO22c} \gabf = \frac{1}{r+1}(\gabf_1 + \cdots + \gabf_{r+1}).\end{equation}
Fix $i\le s$. Then there is a point $\gabf_j$ which does not lie on the hyperplane defined by $A_i\xbf = b_i$. This means $A_i\gabf_j < b_i$, whence $A_i\gabf_j \le  b_i - \varepsilon_0$. Then
\begin{eqnarray} \nonumber A_i  \gabf & = & \frac{1}{r+1}(A_i \gabf_j + \sum_{l\neq j}A_i \gabf_l) \\
\label{EO22b} & \le & \frac{1}{r+1}(b_i - \varepsilon_0 + rb_i) = b_i - \frac{\varepsilon_0}{r+1}.
\end{eqnarray}
Hence $b_i - A_i\gabf \ge \varepsilon_0/(r+1)$, or by (\ref{EO11}), $\varepsilon_\gabf \ge \varepsilon_0/(r+1)$. Using (\ref{EO12})  this implies
$$N_\gabf \le 1 +\frac{r+1}{ \varepsilon_0} \max\{0, \ - c_1, ..., -c_s \} = N_0.$$
By Lemma \ref{O1}, we have $n\gabf \in \inte(\Qcal_n)$ for all $n \ge N_0$.
\end{proof}

Assume now that $\Pcal$ is full-dimensional and  the optimum of the following linear program
 $$ (LP) \ \ \max\{ \dbf^T \xbf|\ A \xbf \le \bbf,\ \xbf \ge \zerobf \},$$
 is finite, where $\dbf = (d_1,...,d_r)^T \in  \Rset^r$. Let $n\ge N_0$. By Lemma \ref{O2}, $\Qcal_n \neq \emptyset $. Hence the optimum $\varphi_n$
 of the following linear program
 $$ (LQ_n) \ \ \max\{ \dbf^T \xbf|\  A \xbf \le n \bbf  + \cbf ,\ \xbf \ge \zerobf \},$$
 is either finite or equal to $ \infty $. Assume that $\varphi_n = \infty $.  Then  $\Qcal_n$ contains  a half-line such that the value $\dbf^T\xbf$ along this half-line tends to $\infty $, i.e.,  there is $\ybf \in \Qcal_n$ and $\vbf$ such that  $\dbf^T \vbf  > 0$ and $\ybf + \alpha \vbf \in \Qcal_n$ for all $\alpha \ge 0$. Then $\Pcal$ also contains a half-line with the same direction $\vbf$. Indeed, since $A_i(\ybf + \alpha \vbf) \le nb_i + c_i$, $A_i\vbf \le \lim_{\alpha \rightarrow \infty }\frac{nb_i + c_i - A_i \ybf}{\alpha } = 0$.  Let $\zbf$  be an arbitrary point of $\Pcal$. Then 
$$A_i(\zbf + \alpha \vbf) = A_i(\zbf )  +  \alpha A_i \vbf \le b_i + 0 = b_i,$$
for all $i\le s$ and $\alpha \ge 0$. This means $\zbf + \alpha \vbf \in \Pcal$. Now we have 
$$ \dbf^T(\zbf + \alpha \vbf )  = \dbf^T\zbf + \alpha \dbf^T \vbf  \longrightarrow \infty ,\ \ \text{when}\ \alpha \longrightarrow \infty ,$$
a contradiction. Therefore we always have $\varphi_n < \infty $.
\vskip0.5cm
 
   The fact that $\varphi_n$ is a linear function of $n$, where $n\gg 0$, is perhaps well-known. However we cannot find a reference with an explicit formulation of this result, so that we include a proof for the sake of completeness. It easily follows from the duality,  cf. the proof of \cite[Lemma on p. 468]{WW}.
 
 \begin{prop} \label{O3}
 Assume that $\Pcal$ is full-dimensional and the optimum $\varphi $ of the  linear program $(LP)$ 
  is finite. Then the optimum $\varphi_n$  of the  linear program $ (LQ_n) $  is a linear function of $n$ of slope $\varphi $, i.e., 
 $$\varphi_n = \varphi n + \varphi_0 ,$$
 for some $\varphi_0 $ and for all $n\gg 0$. 
 \end{prop}
 
 \begin{proof} Let $n\ge N_0$. By the duality theorem for linear programming (see, e.g., \cite[Corollary 7.1g and (25) on page 92]{Sch}), we have
$$\varphi_n = \max\{ \dbf^T \xbf|\  A \xbf \le n \bbf  + \cbf ,\ \xbf \ge \zerobf \} = \min\{ \ybf^T( n \bbf  + \cbf)|\ \ybf \ge \zerobf,\ \ybf^T A \ge \dbf^T\},$$
where $\ybf \in \Rset^s$.  The dual programs have the same feasible region for all $n$. Since  $\varphi_n$ is bounded, it is attained by a vertex of the polyhedron $\{\ybf \ge \zerobf,\ \ybf^T A \ge \dbf^T\}$. Let $\ybf^*_1,..., \ybf_m^*$ be all vertices of this polyhedron. Then
$$\varphi_n = \min_{1\le i \le m}\{((\ybf^*_i)^T\bbf)n + (\ybf^*_i)^T \cbf \}.$$
If $v_*$ is the largest coordinate of intersection points of different lines among $u = ((\ybf^*_i)^T\bbf)v + (\ybf^*_i)^T \cbf,\ i=1,...,m$, then for all $v\ge v_*$, there is one line, say $u = ((\ybf^*_1)^T\bbf)v + (\ybf^*_1)^T \cbf$, which lies below all other lines. Clearly $(\ybf^*_1)^T\bbf $ is the smallest slope which is equal to $\varphi $. This means
$\varphi_n =\varphi n + (\ybf^*_1)^T \cbf$ for all $n\ge N_1 := \max\{N_0, v_*\}$.
 \end{proof}

 \begin{rem} \label{O4}
 1. Assume that $\Pcal \neq \emptyset$. The property  $\Qcal_n \neq \emptyset $  always holds if all $c_i = 0$, since in this case $\Qcal_n = \Pcal_n = n\Pcal$. However, if $c_i\neq 0$ for some $i$, then the assumption that $\Pcal$ is full-dimensional cannot be omitted. 
 \vskip0.3cm
 
 \noindent {\it Example}. Let $\Pcal \subset \Rset^2$ be a segment defined by
 $$\Pcal: \ \ \begin{cases} x_1 + x_2 & \le 2, \\
 - x_1 - x_2 &\le -2, \\
 x_1, x_2 &\ge 0.
 \end{cases}$$
 Let $c_1 = -1, c_2 =0$. Then
 $$\Qcal_n: \ \ \begin{cases} x_1 + x_2 &\le 2n -1, \\
 - x_1 - x_2 & \le -2n, \\
 x_1, x_2 &\ge 0.
 \end{cases}$$
 We have $\Pcal \neq \emptyset$, $\dim \Pcal = 1$, but $\Qcal_n = \emptyset $ for all $n\ge 1$.
 \vskip0.3cm

 2. Assume that $\Pcal$ is full-dimensional and $c_1,...,c_s$ are fixed.  Then $\Qcal_n$ could be empty for some small $n$. However, for $n\gg 0$, $\Qcal_n \neq  \emptyset$. 
 \vskip0.3cm
 \noindent {\it Example}. Let
 $$\Pcal: \ \ \begin{cases} 
 - x_1 - x_2 &\le -2, \\
 x_1 &\le 2,\\
 x_2 &\le 2,\\
 x_1, x_2 &\ge 0.
 \end{cases}$$
 If $c_1= 2m,\ m\ge 2$, and $c_2=c_3 = 0$, then $\Qcal_n = \emptyset $ if and only if $n< m$.
 
 3. From the proof of Proposition \ref{O3} one can give an estimation on the value $N_1$ from which $\varphi_n$ becomes a linear function. 
\end{rem}
 
 \section{Asymptotic behavior of   Integer Programming} \label{IOpt}
 
 In this section we always assume that $a_{ij}, b_i, c_i, d_j \in \Zset$. As usual, let $\dbf = (d_1,...,d_r)^T, \ \xbf = (x_1,...,x_r)^T$ and $n\in \Nset$. Consider the following integer program
  $$ (IQ_n) \ \ \begin{cases} \max \ \ \dbf^T \xbf\\
 \sum_{j=1}^r a_{ij}x_j \le nb_i + c_i\ \ (i=1,...,s);\\
 x_j \in \Nset\ \ ( j = 1,...,r).
 \end{cases}$$
 The corresponding polyhedron of this problem is  $\Qcal_n$.
 In particular, when $c_1=\cdots = c_s = 0$ we get the following integer program
 $$ (IP_n) \ \ \begin{cases} \max \ \ \dbf^T \xbf\\
 \sum_{j=1}^r a_{ij}x_j \le nb_i \ \ ( i=1,...,s);\\
 x_j \in \Nset\ \  (j = 1,...,r),
 \end{cases}$$
 and the corresponding polyhedron is $\Pcal_n = n\Pcal$. 
 
 From Proposition \ref{O3} we immediately get the following sufficient criterion for $M_n$ to be an eventually linear function. Another criterion is given in Corollary \ref{I9b}.
  
  \begin{cor} \label{O3I} Assume that $\Pcal$ is full-dimensional and the optimum of the  linear program
 $ (LP)$  is finite. Assume further that all vertices of $\Qcal_n$ are integral for all $n\ge 1$. Then the maximum $M_n$ of $(IQ_n)$ is a linear function of $n$ for all $n\gg 0$.
 \end{cor}
 
 If the matrix $A$  is totally  unimodular, i.e., if all its subdeterminants  are either $0$ or $\pm 1$, then the second assumption in the above corollary holds. On the other hand,  in general even the maximum $m_n$ of $(IP_n)$ is not an asymptotically linear function of $n$ (see Example \ref{I10}).
 
  If $\alpha $ is a real number, then 
 $$\lfloor \alpha \rfloor  , \ \ \lceil \alpha \rceil \ \ \text{and}\ \  \{ \alpha  \} = \alpha - \lfloor \alpha \rfloor$$
 denote the lower integer part, the upper integer part and the fractional part, respectively, of $\alpha $.
 
 Denote by $A(i_1,...,i_k; j_1,...,j_k)$ the  submatrix of $A$ with elements in the rows $i_1,...,i_k$ and the columns $j_1,...,j_k$, $k\le \min\{r,s\}$, and 
$$D(i_1,...,i_k; j_1,...,j_k) = \det( A(i_1,...,i_k; j_1,...,j_k)).$$
 In the sequel, set
\begin{equation} \label{EO24}
 \Ical = \left\{ (i_1,...,i_k; j_1,...,j_k) \Bigg| 
\begin{array}{l}
 1\le k\le \min\{r,s \}, \\
 1\le i_1< \cdots < i_k \le s,\\ 
 1\le j_1<\cdots < j_k\le r, \\ 
 D(i_1,...,i_k; j_1,...,j_k) \neq 0
 \end{array}
 \right\}.
 \end{equation}
 
 If $\Ibf = (i_1,...,i_k; j_1,...,j_k) \in \Ical$ and $\ubf  \in \Rset^r$,  we denote by $\zbf_{\Ibf,\ubf}$ the solution of the system
 \begin{equation} \label{EO25}
 \begin{cases} A(i_1,...,i_k; j_1,...,j_k) (x_{j_1},...,x_{j_k})^T = (u_{i_1},...,u_{i_k})^T,\\
 x_l =0 \ \forall l\not\in \{j_1,...,j_k\}.
 \end{cases}
 \end{equation}
 
 In the following result we give a sufficient condition to guarantee that the integer program  ($IQ_n$) has a feasible solution. It is a kind of improvement of Lemma \ref{O2}. The intuition for this result is the following: since the data are integers, the slack $\varepsilon_0$ defined in (\ref{EO21}) is not too small. Therefore, one can give an estimation of  $n$ such that the slacks of  the point $n\gabf$ (defined by (\ref{EO22c})) in $\Qcal_n$ are so big, that one of its approximate integer   points still belongs to $\Qcal_n$. This  idea was used in the proof of \cite[Lemma 3.2]{HT3}.  
 
 \begin{lem} \label{I1} Assume that the polyhedron $\Pcal$ is full-dimensional. Let  $\Qcal_{I,n} = \Qcal_n \cap \Nset^r$. Set
 $$a^* := \max_{i,j} |a_{ij}| \ \ \text{and}\ \ c^* = \min_{1\le i\le s} c_i.$$
 Then $\Qcal_{I,n} \neq \emptyset $ for all
 $$n\ge \kappa := \max\{1, \ (r+1)D^2(\frac{r}{2} a^* - c^*) \},$$
 where $D$ is the maximum absolute value of the subdeterminants of the matrix $A$. 
  \end{lem}
 
 \begin{proof} Keep the notation in Lemma \ref{O2} and its proof.  

First we give a lower bound for the number $\varepsilon_0$  defined in (\ref{EO21}).   Note that each $\albf_j$ in Lemma \ref{O2} is a solution of a system of linear equation of the type  (\ref{EO25}) with $(i_1,...,i_k; j_1,...,j_k) \in \Ical$ (see the notation (\ref{EO24})).
  Hence
 \begin{equation} \label{EI11}
 D_1 \albf_j \in \Zset^r, \ \ \text{where}\ \ D_1:= D(i_1,...,i_k; j_1,...,j_k).
 \end{equation}
 One can also assume that each vector $\bebf_{j'}$ of the cone $\cone (\bebf_1,...,\bebf_q)$ in Lemma \ref{O2} is a 
 a  solution of a system of linear equation of the type 
 \begin{equation} \nonumber
 \begin{cases} A(i'_1,...,i'_{l-1}; j'_1,...,j'_{l-1}) (x_{j'_1},...,x_{j'_{l-1}})^T = \zerobf,\\
 x_{j'_l} = 1,\\
 x_m =0 \ \forall m\not\in \{ j'_1,...,j'_l \},
 \end{cases}
 \end{equation}
 where $(i'_1,...,i'_{l-1}; j'_1,...,j'_{l-1}) \in \Ical$   and $j'_l \not\in \{ j'_1,...,j'_{l-1} \}$. Then
 $$  D_2 \bebf_{j'} \in \Zset, \ \ \text{where}\ \ D_2:= D( i'_1,...,i'_{l-1}; j'_1,...,j'_{l-1}).$$
 Hence 
$0\neq D_1D_2 (b_i- A_i (\albf_j + \bebf_{j'}) ) \in \Zset$ for all $i$ such that $A_i (\albf_j + \bebf_{j'}) < b_i$. This implies
 $ |D_1D_2| (b_i- A_i (\albf_j + \bebf_{j'}) ) \ge 1$, whence
 \begin{equation}\label{EI12}
 \varepsilon_0 \ge 1/D^2,
 \end{equation}
Now, let  $\gabf$  be the same point defined in (\ref{EO22c}). Define $\wbf\in \Nset^r$ as follows
 $$w_l = \begin{cases} \lfloor n\gamma_l \rfloor \ \ \text{if}\ \ \{n\gamma_l\} \le 1/2,\\
 \lceil n\gamma_l \rceil \ \ \text{if}\ \ \{n\gamma_l\} > 1/2. 
 \end{cases}$$
 This means $w_l - n\gamma_l \le 1/2$ for all $l$. Fix $n\ge \kappa$. Using (\ref{EO22b}) and (\ref{EI12}), for all $i\le s$, we get
 $$\begin{array}{ll}
 A_i \wbf &= A_i (n\gabf) + A_i (\wbf - n\gabf)\\
& \le n(b_i - \frac{\varepsilon_0}{r+1}) + \frac{r}{2}a^*\\
&\le nb_i - \frac{n}{(r+1)D^2} +  \frac{r}{2}a^*\\
&\le nb_i + c^* \ \ (\text{since}\ n\ge \kappa) \\
&\le nb_i + c_i.
\end{array}$$
Thus $\wbf \in \Qcal_n$.
\end{proof}

\begin{rem} \label{I2} Even in the case $c_1=\cdots = c_r = 0$ the assumption $\Pcal$ being full-dimensional in Lemma \ref{I1} cannot be omitted. For an example, 
$$\Pcal: \ \ \begin{cases} 3(x_1+ \cdots + x_r) =1,\\
x_1,...,x_r \ge 0,
\end{cases}$$
has dimension $r-1$ and $\Pcal_n = n\Pcal$ has an integer point if and only if $n$ is divisible by $3$.
\end{rem}

Gomory \cite{Go} considered a family of  integer programs 
$$ P(\bbf'):  \ \ \begin{cases} 
M(\bbf') & = \max \dbf'^T \xbf',\\
\text{s. t.} & A'\xbf' = \bbf', \ x'_j \in \Nset  \ (j=1,.., s'+r'),
\end{cases}$$
where $A' \in M_{s',s'+r'}(\Zset)$, $\bbf' \in \Zset^{s'}$ and $\rank(A') = s'$. Here $A', \ \dbf'$ are fixed and $\bbf'$ is considered as a vector parameter. Of course, our case $\bbf' = n\bbf + \cbf$ is only a very special case. Let $B$ be an optimal basis for the linear programming relaxation of  $P(\bbf')$.  Without loss of generality we can assume that $B$  consists of the first $s'$ columns, so
$$A' = (B, N),$$
where $N\in M_{s',r'}(\Zset)$. For a vector $\xbf'\in \Rset^{s'+r'}$, we write $\xbf' = (\xbf_B, \xbf_N)$, where $\xbf_B\in \Rset^{s'}$ and $\xbf_N \in \Rset^{r'}$.    If $\bbf'$ is sufficiently deep in the cone $\{\zbf \in \Rset^{s'}|\ B\zbf \ge \zerobf\}$, Gomory gave a formula for an optimal solution to $P(\bbf')$ and showed that this solution is periodic in the columns  $B_j$ of $B$, see \cite[Theorem 5]{Go}. This result was extended by Wolsey to any $\bbf'$ as following:

\begin{thm} {\rm (\cite[Theorem 5]{Go} and \cite[Theorem 1]{Wol})} \label{GoW1} Given fixed $A'$ and $\dbf'$ for which $\{\ybf'^T A \ge \dbf'^T\} \neq \emptyset $, there exists a finite dictionary containing a list of nonnegative integer   vectors $\{\xbf_N^q\}_{q\in Q_B}$ for each dual feasible basis $B$ with the following property: Let 
$$Q_{\bbf'} = \{ q\in Q_B|\  \xbf_B^q (\bbf') = B^{-1}\bbf' - B^{-1}N \xbf_N^q \ge \zerobf \ \text{and integer  } \}.$$
Either $Q_{\bbf'} \neq \emptyset $ implying that $P(\bbf')$ is infeasible, or $Q_{\bbf'} \neq \emptyset $ implying that if $\bar{\dbf}_N^T\xbf_N^{q(\bbf')} = \min_{Q_{\bbf'}}\{ \bar{\dbf}_N^T\xbf_N^{q} \}$,  where $\bar{\dbf} = \dbf_NB^{-1}N -\dbf_N \ge \zerobf$, the vector $(\xbf_B^{q(\bbf')}(\bbf'), \xbf_N^{q(\bbf')})$ is an optimal solution to  $P(\bbf')$.    
\end{thm}

To proceed further, we need the following notion:

\begin{defn} \label{I8}
We say that a  function $f:\ \Nset \rightarrow \Rset$ is a {\it quasi-linear function of period} $t$ if there are finitely many linear functions $f_0, ..., f_{t-1}$, $t\ge 1$, such that $f(n) = f_i(n)$ if $n \equiv  i \nmod t$ for all $0\le i \le t-1$. For short, we  denote $f$ by $(f_i)$, i.e., $f= (f_i)$.

A function $g:\ \Nset \rightarrow \Rset$  is an {\it eventually quasi-linear function}, if it agrees with a quasi-linear function for sufficiently  large $n$.
\end{defn}

From now on,  assume that $\Pcal$ is full-dimensional and that the following linear program 
$$(LP): \ \ \max \{\dbf^T \xbf| \ A\xbf \le \bbf,\ x_j\ge 0\ (j=1,...,r) \},$$
has a finite optimum  $\varphi $.  If the integer  program  ($IP_n$) (resp., ($IQ_n)$) has a feasible solution, then by Proposition \ref{O3}, it also has a finite optimum. In this case, we denote these optima  by $m_n$ and $M_n$, respectively. Otherwise we set $m_n = - \infty $ (respectively, $M_n = -\infty $) or we also say that $m_n$ (respectively, $M_n$) is not (well) defined.  If  $\bbf'$ is sufficiently deep in the cone $\{\zbf  \in \Rset^{s'} |\ B\zbf \ge \zerobf\}$, \cite[Theorem 5]{Go} also claims that the $r'$-vector $\xbf_N^{q(\bbf')}$  is  periodic in the columns  $B_j$ of $B$, i.e., $\xbf_N^{q(\bbf' + B_j)} = \xbf_N^{q(\bbf')}$. In the general case, it is unclear if this property still holds, but from (\ref{EGoW2a}) below it is easy to show that this solution has the  asymptotic periodicity in columns, in the sense that $\xbf_N^{q(\bbf' + (n+1) B_j)} =  \xbf_N^{q(\bbf' + nB_j)}$ for all $n\gg 0$. 
 This periodicity signifies that $M_n$ is an eventual quasi-linear function of $n$. This property is confirmed in  the following consequence of Theorem \ref{GoW1}.

\begin{prop} \label{GoW2} $M_n$ is an eventually quasi-linear function of $n$ of a period at most $D$ - the maximum absolute value of the subdeterminants of the matrix $A$. 
\end{prop}

\begin{proof} By the standard technique, we can reformulate the integer program $(IQ_n)$ in form of $P(n\bbf +\cbf)$ with $A' = (A, I_s)$, where $I_s$ is the unit matrix of size $s$, $\xbf' = (x_1,...,x_r, z_1,....,z_s)^T$, $\dbf' = (d_1,...,d_r, 0, ..., 0)^T$ and $\bbf' = n\bbf + \cbf$. We may assume that $\rank(A') = s$.

 Let $n\ge \kappa $ be an arbitrary integer, where $\kappa $ is  defined in Lemma \ref{I1}. By Lemma \ref{I1}, $P(n\bbf +\cbf)$ is feasible.  By Theorem \ref{GoW1},  for each $n$ there is a dual feasible basis $B$ such that  $Q_{n\bbf +\cbf}\neq \emptyset $. Note that one can choose $\kappa'$ such that $B(n\bbf +\cbf) \ge \zerobf$ for all $n\ge \kappa'$ if and only if $B(\kappa'\bbf +\cbf) \ge \zerobf$, i.e., all vectors $n\bbf+ \cbf,\ n\ge \kappa'$  lie in the same cone. Therefore one can choose the same $B$ for all $n\ge \max\{\kappa , \kappa'\}$. Fix $B$. So, the vector $(\xbf_B^{q(\bbf')}(\bbf'), \xbf_N^{q(\bbf')})$ is well defined. Note that $M_n = M(\bbf')$.  In order to show that  $M_n$ is an eventually quasi-linear  function, it suffices to show the periodicity of $q(\bbf')$, or equivalently, to show the periodicity of the sets $Q_{n\bbf +\cbf}$,  provided $n\gg 0$.

For  any column vector $B_j$ of $B$ we have
\begin{equation} \label{EGoW2a}
\xbf_B^q(\bbf' + B_j) = B^{-1}(\bbf' + B_j) - B^{-1}N \xbf_N^q = \xbf_B^q(\bbf')  + \ebf_j,
\end{equation}
where $\ebf_j$ is the $j$-th basis vector of $\Rset^s$.  We now show that this property implies the periodicity  of  $Q_{n\bbf +\cbf}$ for $n\gg 0$. 

Choose a number $\kappa"$ such that for any $q\in Q_B$  and $n\ge \kappa"$, we have  $\xbf_B^q (n\bbf + \cbf) \ge \zerobf$ if and only if $\xbf_B^q (\kappa" \bbf + \cbf) \ge \zerobf$.  Let $\delta = | \det(B) |$. Then
$$\delta \bbf  = a_1B_1 +\cdots + a_sB_s,$$
for some integers $a_1,...,a_s$.

For any $n\ge \max\{\kappa , \kappa', \kappa"\} + \delta $ and $q\in Q_B$ we have
\begin{itemize}
\item[(i)] $\xbf_B^q (n \bbf + \cbf) \ge \zerobf$ if and only if  $\xbf_B^q ((n-\delta ) \bbf + \cbf) \ge \zerobf$,
\item[(ii)] $\xbf_B^q (n \bbf + \cbf) =  \xbf_B^q ((n-\delta ) \bbf + \cbf + \delta \bbf)  = \xbf_B^q ((n-\delta ) \bbf + \cbf ) + a_1\ebf_1 + \cdots + \abf_s\ebf_s$. 

This means $\xbf_B^q (n \bbf + \cbf) $ is an integer   vector if and only if so is  $\xbf_B^q ((n-\delta ) \bbf + \cbf)$.
\end{itemize}
These two properties imply that $Q_{n \bbf + \cbf} = Q_{(n-\delta ) \bbf + \cbf} $, as required.
\end{proof}

 Note that the property $M_n$ being an eventually quasi-linear function directly follows from a much more general  result in the recent work \cite[ Theorem 3.5(a), Property 3a]{Wo} (also see \cite[Theorem 1.4]{Sh}), where the parameters of integer programs  can be polynomials of one variable. It is also interesting to note that Gomory's formula of an optimal solution to $P(\bbf')$ in Theorem \ref{GoW1} is still used until now, see, e.g., \cite[Theorem 2.9]{RSVH}. On the other hand, in \cite{StT}, a relationship between Integer Programming and Gr\"obner bases was studied, when the cost function varies. 
 
 For our application in Section \ref{CM}, we need to  find a number  $N_*$,  such that $M_n$ becomes  a quasi-linear function for all $n\ge N_*$. Unfortunately, neither results nor proofs in the papers \cite{Sh, Wo} provide an estimation for $N_*$. Analyzing the proof of Proposition  \ref{GoW2},  we would get such an estimation if we could  understand better the complexity of vectors $\{\xbf_N^q\}_{q\in B}$. This is not a trivial task.

 In the rest of this section we give a totally different proof of Proposition \ref{GoW2}, which  gives a way to find such a number $N_*$.  Our proof is mainly combinatorial and based on the module structure of the family of integer points $\{\Qcal_{I,n}\}_{n\ge1}$.  For that purpose, consider the following polyhedra in $\Rset^{r+1}$
$$\tilde{\Pcal} =\left\{ \begin{pmatrix} \xbf \\  y \end{pmatrix}|\ A \xbf - y\bbf \le \zerobf;\ \xbf \ge \zerobf,\ y\ge 0 \right\},$$
and 
$$\tilde{\Qcal} =\left\{  \begin{pmatrix} \xbf \\  y \end{pmatrix}|\ A \xbf - y\bbf \le \cbf;\ \xbf \ge \zerobf,\ y\ge 0 \right\}.$$
Note that $\Ptil$ is a pointed polyhedral cone, $\begin{pmatrix} \vbf \\  n \end{pmatrix} \in \Ptil$ if and only if $\vbf \in \Pcal_n$ and $\begin{pmatrix} \vbf \\  n \end{pmatrix} \in \Qtil$ if and only if $\vbf \in \Qtil_n$.

By \cite[Theorem 6.4]{Sch}, $\Ptil_I := \Ptil \cap \Nset^{r+1}$ forms a finitely generated semigroup, so that the $K$-vector space $K[ \Ptil_I]$ is a Noetherian ring, where $K$ is a field. The set $K[ \Qtil_I]$, where $\Qtil_I := \Qtil\cap \Nset^{r+1}$,  is a finitely generated module over $K[ \Ptil_I]$. The following result is given in the proof of \cite[Theorem 17.1]{Sch}. For convenience of the reader we give a sketch of the proof here.

\begin{lem} \label{I3} Assume that $\Ptil_I \neq \emptyset $ and $\Qtil_I \neq \emptyset $.  Then

 (i) The semigroup ring $K[\Ptil_I]$  is generated by (integer) vectors with all components less than $(r+1)D'$ in absolute value, where $D'$ is the maximum absolute value of the subdeterminants of the matrix $[A\ \bbf]$.

(ii) The module $K[ \Qtil_I]$  is generated over $K[ \Ptil_I]$ by (integer) vectors with all components less than $(r+2)\Delta $ in absolute value, where $\Delta $ is the maximum absolute value of the subdeterminants of the matrix $[A\ \bbf \ \cbf]$.
\end{lem}

\begin{proof} (Sketch):  By Cramer's rule, the polyhedral cone $\Ptil$ admits the following representation
$$\Ptil = \cone (\bebf_1,...,\bebf_q),$$
where $\bebf_1,...,\bebf_q$ are integer   vectors with  each component being a subdeterminant  of $[A\ \bbf]$ - in particular each component is at most $D'$ in absolute  value. Let $\xbf_1,...,\xbf_t$ be the integer   vectors contained in 
$$\begin{array}{ll} \Ecal := & \{\mu_1\bebf_1 + \cdots + \mu_q \bebf_q|\ 0 \le \mu_j < 1 \ (j=1,...,q); \\
 &\ \text{at most}\ r+1 \ \text{of the} \ \mu_j\ \text{are nonzero}  \}.
 \end{array}$$
 Then each component of each $\xbf_j$ is less than $(r+1)D'$ and the set $\{ \bebf_1,...,\bebf_q, \xbf_1, ... ,\xbf_t\} $ forms a basis of the semigroup $\Ptil_I$.
 
Similarly,  $\Qtil$ admits the following decomposition,
$$\Qtil = \conv (\albf_1,...,\albf_p) + \cone (\bebf_1,...,\bebf_q),$$
where $\bebf_1,...,\bebf_q$ are defined as above, and $\albf_1,...,\albf_p$ are vectors with each component being a quotient of subdeterminants of $[A\ \bbf\ \cbf]$.  In particular, each component  of $\albf_i$ is at most $\Delta $ in absolute value. Let $\xbf_1, ... , \xbf_u,\ u> t$, be the integer   vectors contained in  
$$\conv (\albf_1,...,\albf_p) + \Ecal.$$
Then each component of each $\xbf_j,\ j\le u$, is less than $\Delta + (r+1)D' \le (r+2)\Delta $, and the set $\{\xbf_{t+1}, ... , \xbf_u \}$ forms a basis of  the module $K[\Qtil_I]$ over $K[\Ptil_I]$.
\end{proof}

  For short, we also set $\Pcal_{I,n} :=  \Pcal_n \cap \Nset^r$.  Note that
 \begin{equation} \label{EI1}  \Pcal_{I,n} + \Pcal_{I,m} \subseteq \Pcal_{I,n+m}  \ \ \text{and} \ \  \Pcal_{I,n} + \Qcal_{I,m} \subseteq \Qcal_{I,n+m},
 \end{equation}
 for all  numbers $n,m \in \Nset$. This implies that the ring $K[\Ptil_I]$ admits the so-called $\Nset$-graded structure, namely
$$K[\Ptil_I] = \oplus_{n\ge 0} K[\Pcal_{I,n}] ,$$
where $\deg(\ubf) = n$ if  $\ubf \in \Pcal_{I,n}$. Similarly,   $K[\Qtil_I]$ is a graded module over $K[\Ptil_I]$:
$$K[\Qtil_I] = \oplus_{n\ge 0} K[\Qcal_{I,n}] ,$$
where $\deg(\vbf ) = n$ if  $\vbf \in \Qcal_{I,n}$.  This interpretation allows us to relate elements of $\Pcal_{I,n}$ and $\Qcal_{I,n}$ to elements of other sets $\Pcal_{I,m}$ and $\Qcal_{I,m}$ with smaller indices $m$.  Below are some elementary properties of $m_n$ and $M_m$.

  From (\ref{EI1}) we get the following inequalities for all $i,j>0$:
\begin{equation} \label{EI2} m_i + m_j \le m_{i+j} \ \ \text{and}\ \  m_i + M_j \le M_{i+j}.
\end{equation}

\begin{lem} \label{I4}     Assume that $\Pcal_{I,n} \neq \emptyset $. Then
$  m_n \le \varphi n$.  Moreover, there is $j_* \le D$ such that $m_{j_*} = \varphi j_*$.
\end{lem}

\begin{proof}  The inequality $m_n \le \varphi n$ is trivial. Since $(LP)$ has a finite optimum, there is a vertex $\albf $ of $\Pcal$ such that $\dbf^T \albf = \varphi $.  Note that all components of $\albf$ are  rational numbers of  a common denominator which is a subdeterminant of $A$  - in particular, there is $j_* \le D $ such that $j_*\albf \in \Nset^r \cap \Pcal_{j_*} = \Pcal_{I, j_*}$. Hence $m_{j_*} = \dbf^T (j_*\albf ) = j_* \varphi $.
\end{proof} 

As we see from Lemmas \ref{I1}, \ref{I3} and \ref{I4}, the set of indices:
\begin{eqnarray} \label{EI51}
\Jcal_{\max} & := & \{ n< (r+1)D'|\  \Pcal_{I,n} \neq \emptyset \ \text{and}\ \ \frac{m_n}{n} = \varphi \} \\
&  :=  & \{ j_1 < \cdots < j_p < (r+1)D'  \}, \nonumber
\end{eqnarray}
contains $j_* $, where $j_*\le D$ is defined in Lemma \ref{I4}.  In particular $\Jcal_{\max} \neq \emptyset $ and $j_1\le D$. Let 
$$\delta := \gcd(j_1,...,j_p).$$
  This number is in general much less than $D$ and will be proved to be an upper bound for the period of $M_n$. By Schur's bound on the Frobenius number, see \cite{B}, we have

\begin{lem} \label{I5} All multiples of $\delta $ bigger or equal to $ \max\{ (j_1-1)(j_p-1),\ j_1\}$ belong to the numerical semigroup $S:= \Nset j_1 + \cdots + \Nset j_p \subseteq \Nset \delta $.
\end{lem}

The following simple result shows that the upper bound of $m_n$ in Lemma \ref{I4} is attained in many places. 
\begin{lem} \label{I6}
If $n\in S$, then $\Pcal_{I,n} \neq \emptyset $ and $m_n = \varphi n$.
\end{lem}

\begin{proof} Assume that $n= \sum n_i$, $n_i\in \Jcal_{\max}$. Let $\ubf_i \in \Pcal_{I, n_i}$ such that $ \dbf^T \ubf_{n_i} = m_{n_i} = \varphi n_i$ (see (\ref{EI51})). Then $\sum \ubf_i \in \Pcal_{I,n}$ and $\dbf^T(\sum \ubf_i) = \varphi n$, whence $m_n = \varphi n$ (by Lemma \ref{I4}).
\end{proof}

\begin{rem} \label{I9a} We list here some further properties of $m_n$, $M_n$, $\Pcal_{I,n}$ and $\Qcal_{I,n}$, which immediately follow from Lemma \ref{I3}, Lemma \ref{I6} and (\ref{EI1}).
\begin{itemize}
\item[(i)] If  $\ubf_i \in \Pcal_{I, n_i}$ ($i=1,2$) and $\vbf \in \Qcal_{I,n}$, then $\deg(\ubf_1 + \ubf_2) = \deg(\ubf_1) + \deg(\ubf_2)$ and $\deg(\ubf_1 + \vbf) = \deg(\ubf_1) + \deg(\vbf)$.

\item[(ii)] If $M_n$ is well defined and $j\in S$, then $M_{n+j}$ is well defined. This is applied to $j= mj_1$ for any $m\ge 0$.

\item[(iii)] If  $m_n = \dbf^T \ubf$ for some $\ubf = \ubf_1 + \ubf_2 \in \Pcal_{I,n}$ with $\ubf_i \in \Pcal_{I, n_i}$, then  $m_{n_i} = \dbf^T \ubf_i$ for $i= 1,2$.

Indeed, $m_n =  \dbf^T \ubf =  \dbf^T \ubf_1 + \dbf^T \ubf_2 \le m_{n_1} + m_{n_2}$.  The reverse inequality follows from (\ref{EI2}).

Similarly, if  $M_n = \dbf^T \vbf$ for some $\vbf = \ubf + \vbf' \in \Qcal_{I,n}$ with $\ubf \in \Pcal_{I, n_1}$ and $\vbf' \in \Qcal_{I, n_2}$, then  $m_{n_1} = \dbf^T \ubf$, $M_{n_2} = \dbf^T \vbf'$.

\item[(iv)] Assume that $M_n = \dbf^T( \ubf + \vbf)$, where  $\ubf  \in \Pcal_{I, n_1}$ and $\vbf \in \Qcal_{I,n_2}$. Assume further that  $n_1 = j n_1'$ for some $j\in S$ (defined in Lemma \ref{I6}). Then we may assume that $\ubf = n_1' \ubf'$, where $\ubf' \in \Pcal_{I,j}$ with $\dbf^T \ubf' =  \varphi j$. In particular this is applied to $j_1$.

By Lemma \ref{I6}, $m_{n_1} = \varphi n_1,\ m_j = \varphi j =\dbf^T \ubf'$ for some $\ubf' \in \Pcal_{I,j}$.  By (iii),   $m_{n_1} = \dbf^T \ubf = n'_1 \dbf^T \ubf' $. So, we can replace $\ubf + \vbf$ by $n'_1\ubf' + \vbf$.

\item[(v)]  Let $\Jcal$ be the set of generating degrees of  the semigroup ring $K[\Ptil_I]$.  By Lemma \ref{I4},    $j_1 \in \Jcal$; and by Lemma \ref{I3}(i), all elements in $\Jcal$ are at most $(r+1)D' -1$.  Moreover, if $\ubf\in \Pcal_{I,n}$ for some $n$, then $\ubf$ can be expressed as a sum of some elements $\ubf_i \in \Pcal_{I , \deg(\ubf_i)}$ with $\deg(\ubf_i) \in \Jcal$.
\end{itemize} 
\end{rem}

We can now formulate and prove the main result of this section. 

\begin{thm} \label{I9} Let $D, \ D'$ and $\Delta $ be  the maximum absolute value of the subdeterminants of the matrices $A$, $[A\ \bbf]$ and $[A\ \bbf\ \cbf]$, respectively. Set 
$$N_*:= (r+1)D'(D-1) + (r+2)\Delta - D .$$
Assume that $\Pcal$ is full-dimensional and the linear program  ($LP$) has a finite maximum $\varphi $. Then the maximum $M_n$ of ($IQ_n$) is a quasi-linear function of $n$ with the slope $\varphi $ and period $\delta $ (defined before Lemma \ref{I5}) for all $n\ge N_* $.
\end{thm}

\begin{proof}  The idea of the proof is the following:  Assume that $M_n = \dbf^T\vbf$ for some $\vbf \in \Qcal_{I,n}$. By Lemma \ref{I3} and Remark \ref{I9a}(v), $\vbf$ can be expressed as a sum of one element from $\Qcal_{I,h}, \ h<(r+2)\Delta $, and some other elements from $\Pcal_{I,j}, \ j\in J$. Since these degrees  $j$ are bounded, most of the elements in such an expression  can be grouped into one element, say $\ubf$, in $\Pcal_{I,m j_1}$ for some $ m\in \Nset$.  This element can be replaced by a multiple $m \ubf_0$ of $\ubf_0 \in \Pcal_{I, j_1}$ (by Remark \ref{I9a}(iv)). If we take $m$ as large as possible, then the  sum $\vbf_0$ of the rest of the elements in an expression of $\vbf$ gives an element of a bounded degree. Hence this degree, say $t$, must be repeated when $n\gg 0$. This implies that $M_n$ is ``nearly''   quasi-linear. In order to show that it is really quasi-linear, we should make this degree $t$ periodically unchanged. We can achieve this by adding to $\vbf_0$ a  multiple of $\ubf_0$ from $m\ubf_0$ (this explains why we need to go down to elements of small degrees).

We now give the detail of the proof.
\vskip0.5cm

\noindent {\it Claim 1}.  Assume that $M_n$ is defined (i.e., $\Qcal_{I,n} \neq \emptyset )$ and $M_n = \dbf^T \vbf$ for some $\vbf \in \Qcal_{I,n}$.  Then one can assume that 
$$\vbf = m\ubf_0 + \vbf_0,$$
for some $m\in \Nset$, $\ubf_0 \in \Pcal_{I, j_1}, \  \vbf_0 \in \Qcal_{I, \deg(\vbf_0)}$ and $\deg(\vbf_0) \le N_*$. 

\vskip0.3cm

\noindent {\it Proof of Claim 1}:   By Lemma \ref{I3}(ii) and Remark \ref{I9a}(v), there are vectors $\ubf_i \in \Pcal_{I, \deg(\ubf_i)}$, $\deg(\ubf_i)\in \Jcal$ ($i=1,...,l$) and $\vbf_h \in \Qcal_{I, h}$ such that 
\begin{eqnarray}
\label{EI92} \vbf&=&  \sum_{i=1}^l \ubf_i + \vbf_h,\\
\nonumber \text{where}\  h & <  & (r+2)\Delta \  \text{and}\ n = \sum_{i=1}^l \deg(\ubf_i) + h. 
\end{eqnarray}
Here $l\ge 0$ and $l=0$ means that there is no element $\ubf_i$ in the sub-sum of (\ref{EI92}). Let $U =  \sum_{i=1}^l \ubf_i$.  Using the containment 
$\Pcal_{I,i} + \Pcal_{I,j} \subseteq \Pcal_{I, i+j}$ we can write $U$ in the form
$$U = \ubf + \ubf_{i_1} + \cdots + \ubf_{i_{l'}},$$
where $\ubf \in \Pcal_{I,m j_1}$ for some $m \ge 1$ or $\ubf = \zerobf$ and $1\le i_1<  \cdots < i_{l'} \le l$ ($l'\ge 0$). Choose  such  a representation with the smallest possible $l'$. 

If $l'\ge j_1$, then consider the sums of degrees 
\begin{equation}\label{EI93} \deg(\ubf_{i_1}),\ \deg(\ubf_{i_1}) + \deg(\ubf_{i_2}), ...,  \deg(\ubf_{i_1})+ \cdots +\deg(\ubf_{i_{j_1}}). \end{equation}

If one of these numbers, say $\deg(\ubf_{i_1})+ \cdots +\deg(\ubf_{i_j}) = j_1 m'$ for some $m' >0$, where $1\le j\le j_1$, then taking
$$\ubf' := \ubf + \ubf_{i_1} + \cdots + \ubf_{i_j} \in \Pcal_{I, j_1(m+m')},$$
 as a new $\ubf$, we get
$$U =  \ubf' + \ubf_{i_{j+1}} + \cdots + \ubf_{i_{l'}},$$
a contradiction to the minimality of $l'$. 

So, all  numbers in (\ref{EI93}) are not divisible by $j_1$. But then two of these sums, say $\deg(\ubf_{i_1})+ \cdots +\deg(\ubf_{i_j})$ and $\deg(\ubf_{i_1})+ \cdots +\deg(\ubf_{i_{j'}})$, where $1\le j < j' \le j_1$, are congruent modulo $j_1$. Hence, $\deg(\ubf_{i_{j+1}})+ \cdots +\deg(\ubf_{i_{j'}})= j_1 m"$ for some $m" >0$. Taking
$$\ubf" := \ubf + \ubf_{i_{j+1}} + \cdots + \ubf_{i_{j'}} \in \Pcal_{I, j_1(m+ m")},$$
as a new $\ubf$, we get 
$$U = \ubf" + \ubf_{i_1} + \cdots + \ubf_{i_j} + \ubf_{i_{j'+1}} + \cdots + \ubf_{i_{l'}},$$
again a contradiction to the minimality of $l'$. 

Summing up,  we always have $l'\le j_1 -1$. Let
$$\vbf_0 = \vbf_h + \ubf_{i_1} + \cdots + \ubf_{i_{l'}}.$$
Then, by (\ref{EI92}), we have
$$\begin{array}{ll}   \deg(\vbf_0) & = \deg(\vbf_h ) + \deg( \ubf_{i_1}) + \cdots + \deg(\ubf_{i_{l'}}) \\
& \le (r+2)\Delta - 1 + (j_1-1) ((r+1)D'-1) \\
&\le (r+2)\Delta - 1 + (D-1) ((r+1)D'-1) \\
& = (r+1)D'(D-1) + (r+2)\Delta  - D \\
& = N_*.
\end{array}$$
Since $\vbf = \ubf + \vbf_0$ and $\deg(\ubf) = mj_1$, by Remark \ref{I9a}(iv),  we can assume that $\ubf = m\ubf_0$ for some  $\ubf_0\in \Pcal_{I,j_1}$, which yields $\vbf = m \ubf_0 + \vbf_0$, as required. 
\vskip0.5cm

\noindent {\it Claim 2}.  Let $k$ be a number such that $N_* \le k < N_*+ j_1 $. Let $n = n_1j_1 + k$, where $n_1\ge 1$.  Assume that  $M_n$ is well defined. Then  $M_k$ is well defined, and  $$M_n = \varphi n +  M_k - \varphi k .$$

\vskip0.3cm
\noindent {\it Proof of Claim 2}:   By Claim 1, one can write $M_n = \dbf^T(m\ubf_0 + \vbf_0)$ for some $\ubf_0 \in \Pcal_{I, j_1}$, $\vbf_0\in \Pcal_{I, h}$ and $h\le N^*$. 
Since  $n = mj_1 + h = n_1j_1 + k$ and $h\le N_*\le k$, we have $k - h = (m-n_1)j_1 \ge 0$. By Remark \ref{I9a}(ii), $M_k$ is well defined. Note that $\deg( (m-n_1)\ubf_0 + \vbf_0 ) = (m-n_1)j_1 + h = k$.
Since $M_n = \dbf^T(m\ubf_0 + \vbf_0)$ and $m\ubf_0 + \vbf_0 = n_1\ubf_0 + [(m-n_1)\ubf_0 + \vbf_0]$,
by Remark \ref{I9a}(iii), 
$$M_n = m_{n_1j_1}+ M_k = \varphi n_1 j_1 + M_k = \varphi n - \varphi k + M_k.$$
\vskip0.3cm
Claim 2 already says that $M_n$ is  a quasi-linear function of period $j_1$  for $n\gg 0$. The next claim shows that $M_n$ is  a quasi-linear function of period $j_1$  for  all $n\ge N_*$.
\vskip0.5cm

\noindent {\it Claim 3}. If $n\ge N_*$, then $M_n$ is well defined. Moreover, for any $n,m\ge N_*$  such that $n \equiv m  (\nmod  j_1)$, we have   $$M_n =   M_m + \varphi (n -  m) .$$
\vskip0.3cm

\noindent {\it Proof of Claim 3}:  For the first statement: If  $n\ge \kappa$, then $M_n$ is well defined by Lemma \ref{I1}. So, we may assume that  $n< \kappa$.  Assume that $n= n_1j_1 + k$ for some $N_*\le k< N_*+ j_1$. Let $n'= n'_1j_1 + n =  (n_1' +n_1)j_1 + k$ with $n_1'$ large enough, such that $n'> \kappa$. Since $M_{n'}$ is well defined, by Claim 2, we get that $M_k$ is well defined, i.e., $\Pcal_{I,k} \neq \emptyset $.   By Remark \ref{I9a}(ii),    $M_n$ is well defined.

For the second statement, assume that $n \equiv m \equiv k (\nmod  j_1)$,  for some $N_*\le k< N_*+ j_1$. By Claim 2, $M_n =  \varphi n + M_k - \varphi k$ and $M_m =  \varphi m + M_k - \varphi k$. Hence $M_n =   M_m + \varphi (n -  m)$. 
\vskip0.5cm

Claim 3 already shows that $M_n$ is a quasi-linear function of  period $j_1 $  for all $n\ge N_*$. It remains to show that one can take $\delta $ as  a period of this quasi-linear function. 
 It suffices to show:
 
\vskip0.5cm
  {\it Claim 4}.  For all $n\ge N_*$ we have $M_{n+\delta } - M_n = \varphi  \delta $.
 \vskip0.5cm
  
 Indeed,  fix a larger number $N\gg 0$ so that  $(Nj_1-1)\delta , \ (Nj_1+ 1)\delta \in S$, where  $S$ is defined in Lemma \ref{I5}.
 By Claim 3, we have
 $$M_{(n+\delta ) + Nj_1\delta - \delta } = M_{n + Nj_1\delta }=   Nj_1\delta \varphi + M_{n}.$$
 On the other hand, using (\ref{EI2}) and Lemma \ref{I6}, we also have
 $$M_{(n+\delta )+ Nj_1\delta - \delta } \ge M_{n + \delta } + m_{Nj_1\delta - \delta} = M_{n+\delta } + (Nj_1\delta - \delta) \varphi .$$
 Hence 
 $$M_{n+\delta } - M_{n} \le \varphi \delta .$$
 Using the same argument to $M_{n + N j_1\delta + \delta  }$ we then get
 $$M_{n} + (N j_1\delta + \delta )\varphi  \le M_{n + N j_1\delta + \delta  } = M_{n + \delta + N j_1\delta } = (N j_1\delta )\varphi + M_{n +\delta },$$
 whence
 $$M_{n+ \delta } - M_{n} \ge \varphi \delta .$$
 So we must have $M_{n+\delta } - M_{n} = \varphi \delta $, which completes the proof of Claim 4 and the theorem as well.
\end{proof}
The following consequence of Theorem \ref{I9} gives a criterion for $m_n$ (resp., $M_n$) to be an eventually linear  function.

\begin{cor} \label{I9b} Assume that $\Pcal$ is full-dimensional and the linear program  ($LP$) has a finite maximum $\varphi $. Then
 \begin{itemize}
\item[(i)] $m_n$ is an eventually linear  function if and only if $\delta =1$. In this case $m_n = \varphi n$ for $n\gg 0$.
\item[(ii)] If $\delta =1$, then $M_n$ is an eventually linear  function. The converse does not hold.
\end{itemize}
\end{cor}

\begin{proof} If $\delta =1$, then by Theorem \ref{I9}, both $m_n$ and $M_n$ are eventually linear  functions (quasi-linear functions  of period 1).

Assume that $m_n$ is an eventually linear  function.  By Lemma \ref{I6}, $m_{kj_1 } =\varphi  kj_1 $ for all $k\ge 1 $. Hence we must have $m_n = \varphi n$ for all $n\gg 0$. Assume that $\delta >1$. Fix $n \gg 0$ such that $n$ is not divisible by $\delta $ and $\varphi n = m_n =  \dbf^T\ubf$ for some $\ubf \in \Pcal_{I,n}$. By Remark \ref{I9a}(v), $\ubf = \sum \ubf_i$, where $l_i := \deg(\ubf_i) \in \Jcal$. Since $n$ is not divisible by $\delta $, there is one degree, say $l_1$, is not divisible by $\delta $. Then $l_1 \not\in \Jcal_{max}$. That means $m_{l_1} < \varphi l_1$. By Lemma \ref{I4},  $ \dbf^T\ubf_i \le \varphi  l_i$. All these imply that $\dbf^T\ubf = \sum \dbf^T\ubf_i < \varphi \sum l_i = \varphi n$, a contradiction. Hence $\delta =1$. 

Finally, let us consider the following integer programs
$$\begin{cases} M_n &= \max (x_1 + x_2)\\
4x_1 + 2x_2 &\le 3n + 1,\\
-2x_1 & \le -n.
\end{cases}$$
In this case $c_1 = 1, c_2= 0$. The polytope $\Pcal$ has three vertices $\vbf_1= (1/2, 0), \ \vbf_2 = (3/4, 0), \ \vbf_3 = (1/2, 1/2)$. The maximum $\varphi =1$ is attained at only the vertex $\vbf_3$. Hence, we have $\delta  =2$. However,  $M_n = n$ for all $n\ge1$.
\end{proof} 

We know by Lemma \ref{I1} that  $\Qcal_{I,n}\neq \emptyset $ for all $n\gg 0$, provided that $\Pcal$ is full-dimensional.  However, the following example shows that  $\Pcal_{I,n} \neq \emptyset $ does not imply that $\Pcal_{I, n+1} \neq \emptyset $. Proposition \ref{GoW2} as well as Theorem \ref{I9} say that the period of $M_n$ is bounded by $D$. The following example also shows that this bound is sharp even for $m_n$.

\begin{exm} \label{I10} Let $p_1< \cdots < p_r$ be relatively prime positive integers ($r\ge 2$). Consider the following linear and integer programs:
\begin{equation} \label{EI101}
\begin{cases} \max ( x_1 + \cdots + x_r)\\
p_ix_i \le 1\ \  (i=1,...,r),\
p_1x_1 + \cdots + p_rx_r \ge 1,\\
x_i \ge 0\ \ (i=1,...,r),
\end{cases}
\end{equation}
and
\begin{equation} \label{EI102}
\begin{cases} \max (x_1 + \cdots + x_r)\\
p_ix_i \le n\ \  (i=1,...,r),\
p_1x_1 + \cdots + p_rx_r \ge n,\\
x_i \in \Nset\ \ (i=1,...,r).
\end{cases}
\end{equation}
Then the polyhedron $\Pcal$ of (\ref{EI101}) is full-dimensional and (\ref{EI101}) has the maximum $\varphi = \sum_{i=1}^r 1/p_i$. The polyhedron of the linear relaxation of (\ref{EI102}) is $\Pcal_n$. Let $x_i^0 := \lfloor n/p_i \rfloor > n/p_i -1$. Then 
$$\sum_{i=1}^rp_ix_i^0 > rn - (p_1+\cdots + p_n) \ge n,$$
for all $n>(p_1+\cdots + p_r)/(r-1)$. This means (\ref{EI102}) has a feasible solution for all $n>(p_1+\cdots + p_r)/(r-1)$, while it is clear that (\ref{EI102}) has no feasible solution for all $n< p_1$. We also have $(1,0,...,0) \in \Pcal_{I, p_1}$, $(1,1,0,...,0) \in \Pcal_{I, p_1+p_2}$, ... However, if $p_2 \ge p_1 + a$ and $p_1\ge a$, where $a\ge 2$ is an integer, then $P_{I,n} = \emptyset $ for all $p_1+ 1 \le n\le p_1+ a-1$. This shows that $\Pcal_{I,n} \neq \emptyset $ does not imply that $\Pcal_{I,n+1} \neq \emptyset $.

Note that if $\Pcal_{I,n} \neq \emptyset $, then the maximum of (\ref{EI102}) reaches at the point $(x_1^0,...,x_r^0)$ and the maximum is equal to
$$m_n = \sum_{i=1}^r \lfloor n/ p_i \rfloor =  (\sum_{i=1}^r  1/ p_i )n - \sum_{i=1}^r \{ n/ p_i \}.$$
This function is a quasi-linear function of period $p_1\cdots p_r$ for $n>  (p_1+\cdots + p_r)/(r-1)$. Note that in this case $D = p_1\cdots p_r$.
\end{exm}

The next example shows that $\Qcal_{I,n} \neq \emptyset $ only if $n$ is quite big. Hence the lowest number from which $M_n$ becomes a quasi-linear function of $n$ could be also quite big. In this example  it is $D$.

\begin{exm} \label{I11}
Let $a$ be a positive integer. Consider the following linear and integer programs:
\begin{equation} \label{EI111}
\begin{cases} \max (x_1 + \cdots + x_r)\\
ax_1 \le 1,\
ax_2 - x_1 \le 0,
..., 
ax_r - x_{r-1} \le 0,\
- x_r \le 0,\\
x_i \ge 0 \ \ (i= 1,...,r),
\end{cases}
\end{equation}
\begin{equation} \label{EI112}
\begin{cases} \max (x_1 + \cdots + x_r)\\
ax_1 \le n,\
ax_2 - x_1 \le 0,
... ,
ax_r - x_{r-1} \le 0,\
- x_r \le 0,\\
x_i \in \Nset \ \ (i= 1,...,r),
\end{cases}
\end{equation}
and
\begin{equation} \label{EI113}
\begin{cases} \max (x_1 + \cdots + x_r)\\
ax_1 \le n,\
ax_2 - x_1 \le 0,
...,
ax_r - x_{r-1} \le 0,\
- x_r \le -1,\\
x_i \in \Nset \ \ (i= 1,...,r).
\end{cases}
\end{equation}
That means we consider the case $c_1 = \cdots = c_{r-1}= 0$ and $c_r = -1$. The  polyhedra of  (\ref{EI111}) and of the linear  relaxations of (\ref{EI112}) and (\ref{EI113}) are $\Pcal$, $\Pcal_n$ and $\Qcal_n$, respectively. Then $\Pcal$ is full-dimensional, all $\Pcal_{I,n} \neq \emptyset $ as they contain the point $\zerobf$. However it is easy to check that $\Qcal_{I,n} \neq \emptyset $ if and only if $n\ge a^r$. The maximum of (\ref{EI113}) is equal to
$$M_n = \lfloor n/a \rfloor +  \lfloor n/a^2 \rfloor + \cdots +  \lfloor n/a^r  \rfloor,$$
which is a quasi-linear function of slope $\sum_{i=1}^r 1/a^i$ and period $a^r$.
\end{exm}

The last example shows that $M_n$ can be an eventually linear function, even if $D$ is a big number.

\begin{exm} \label{I12}
Consider the following integer programs
$$\begin{cases} M_n =  \max (x_1 +  x_2)\\
x_1 + x_2 \le n + c_1,\\
-4x_1 \le -n +c_2,\\
-3x_2 \le -n +c_3,\\
x_1, x_2 \in \Nset.
\end{cases}$$
The polytope $\Pcal $ in this example has three vertices: $\vbf_1= (1/4, 1/3), \ \vbf_2 =(1/4, 3/4),$  $\vbf_3 = (2/3, 1/3)$. The maximum  $\varphi = 1$ is attained at $\vbf_2$ and $\vbf_3$. $m_1$ is not defined (that means, equal $-\infty $),  $m_2$ is attained at the point $(1,1)$ which is not a vertex of $\Pcal_2$. The semigroup $S$ contains  all numbers $2, 3, 4$. In particular $\delta = 1$ and $M_n$ is an eventually linear function for all $c_1,c_2,c_3$. In this example $D= 12$.
\end{exm}

\section{Stability of the Castelnuovo-Mumford regularity} \label{CM}

 In this section we study a problem in  Commutative Algebra. First let us recall the notion of  the Castelnuovo-Mumford regularity.  Let $R := K[X_1,...,X_r]$ be a polynomial ring over a field $K$. Consider the standard grading in $R$, that is $\deg(X_i) = 1$ for all $i=1,...,r.$  Let $ \mfr := (X_1,...,X_r)$. For a finitely generated graded $R$-module $E$, set 
\begin{equation} \label{Ereg1} a_i(E) := \sup \{t \mid H_{\mfr}^i (E)_t \ne 0\}, \end{equation}
where $H_{\mfr}^i (E)$ is the local cohomology module with the support  $\mfr$.  The {\it Castelnuovo-Mumford regularity} of   $E$ is defined by
\begin{equation} \label{Ereg2} \reg(E)= \max\{a_i(E)+i\mid 0\le   i \le  \dim E\}. \end{equation}

Another way to define the Castelunovo-Mumford regularity is to use the graded minimal free resolution of $E$. Since this invariant bounds the maximal degrees of all generators of syzygy modules, one can use it as a measure of the cost of extracting information about a module $E$. Beyond its interest to algebraic geometers and commutative algebraists,  the Castelnuovo-Mumford regularity ``has emerged as a measure of the  complexity of computing Gr\"obner bases'', see the itroduction of \cite{BMum}. The interested readers are referred to \cite{BMum} for a survey and to  the book \cite{Ei} for some basic properties and results on the Castelunovo-Mumford regularity.

Let $I$ be a proper homogeneous ideal of a polynomial ring $R$. Then, $\reg(I)$ and $\reg(R/I)$ are well defined, and $\reg(I) = \reg(R/I) + 1$. 

The {\it integral closure} of  $I$  is the set of elements $x$ in $R$ that satisfy an integral relation
$$x^n + \alpha_1x^{n-1} + \cdots + \alpha_{n-1}x + \alpha_n = 0,$$
where $\alpha_i \in I^i$ for $i = 1,\ldots, n$. This is again a homogeneous ideal and is denoted by $\overline I$.

The Castelnuovo-Mumford regularities of $R/I^n$ and $R/\Inba$  could be huge numbers for a fixed number $n$. However, when $n\gg 0$, a striking result of  Cutkosky,  Herzog and Trung,  and  independently of Kodiyalam, says that both $\reg(R/I^n)$ and $\reg(R/\Inba)$ are linear functions of $n$  (see \cite[Theorem 1.1]{CHT} or  \cite[Theorem 5]{Ko}).  In order to understand when these functions become linear functions, let us introduce indices of stability of   Castelnuovo-Mumford regularities.

\begin{defn} \label{CMst} i) Assume that $\reg(R/I^n) = an + b$ for all $n\gg 0$. Set
$$\regst(I) = \min\{t\ge 1|\  \reg(R/I^n) = an + b \ \forall n\ge t\}.$$
ii) Similarly, assume that $\reg(R/\Inba) = cn + d$ for all $n\gg 0$. Set
$$\regb(I) = \min\{t\ge 1|\  \reg(R/\Inba) = cn + d \ \forall n\ge t\}.$$
\end{defn}

It is of great interest to give a bound for $\regst(I)$ and $\regb(I)$.  However, until now  only special cases were treated. Even in the case of $\mfr$-primary ideals the known bounds on $\regst(I)$  are not explicitly defined (see \cite{EU, C}). Explicit bounds on $\regst (I)$ are given  in the case of $\mfr$-primary monomial ideals and some other monomial ideals. 

Recall that a monomial ideal is an ideal generated by monomials. In this section we  can use  results in Section \ref{IOpt} to provide an explicit bound on $\regb(I)$ for any monomial ideal $I$. 

In order to do that we need to recall some notation and  known results.   Given a (column) vector $\albf \in \Nset^r$ we write $\Xbf^\albf := X_1^{\alpha_1}\cdots X_r^{\alpha_r}$.

\begin{defn} \label{NP} Let $I$ be a monomial ideal of $R$. We define
\begin{enumerate}
\item  For a subset $A \subseteq R$, the exponent set of $A$ is $E(A) := \{\albf \mid \Xbf^{\albf}\in A\} \subseteq  \Nset^r$.
\smallskip
\item The Newton polyhedron of $I$ is $NP(I) := \conv\{E(I)\}$, the convex hull of the exponent set of $I$ in
the space $\Rset^r$.
\end{enumerate}
\end{defn}

The integral closure of  a monomial ideal $I$  is a monomial ideal as well. Using the notion of  Newton polyhedron, we can geometrically describe $\overline{I}$  (see \cite{RRV}):
\begin{equation}\label{EN1}
E(\overline I) = NP(I) \cap \Nset^r = \{\albf \in \Nset^r \mid \Xbf^{n\albf} \in I^n \text{ for some } n\geqslant 1\}.
\end{equation}
\begin{equation}\label{EN2}
NP(I^n) =nNP(I)= n\conv\{E(I)\} +\Rset_{+}^r \ \text{for all}\ n\geqslant 1.
\end{equation}

\vskip0.3cm

\begin{figure}[ht]
\setlength{\unitlength}{0.5cm}
\begin{picture}(16,10)
\put(-2,8){\makebox(0,0){ \bf Fig. 1: Newton polyhedron}}
\put(0,6){\makebox(0,0) {$I = (X_1X_2^4, X_1^3X_2^2, X_1^5X_2)$}} 
\put(0,4){\makebox(0,0){ $\begin{array}{l} \text{The hole means the point} \\
\text{does not belong to}\  E(I),\\ \text{or equivalently,}\  X_1^2X_2^3 \in \bar{I}\setminus I. \end{array}$}}
\hspace{4cm}
\includegraphics[width=6cm,height=5cm]{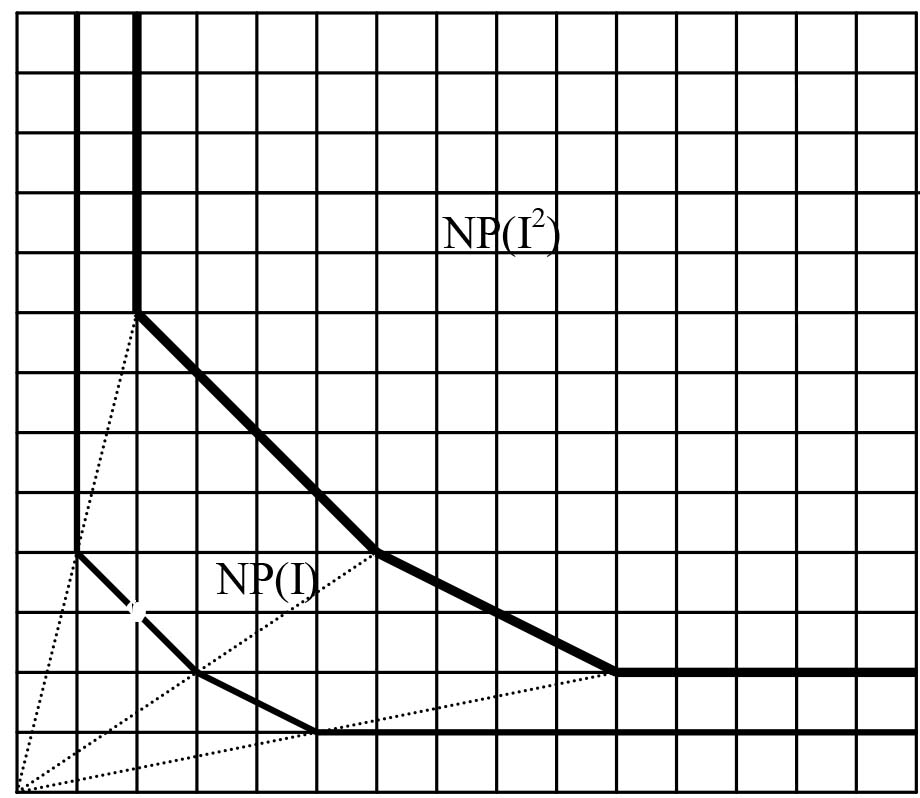} 
 \end{picture}
 \end{figure}
 
The above equalities say that (exponents of)  all monomials  of $\overline{I}$ form the set of  integer points in $NP(I)$ (while we do not know which points among them do not belong to $I$), and the Newton polyhedron $NP(I^n)$  of $I^n$ is just a multiple of $NP(I)$.  This is one of the reasons why it is easier to work with integral closures of powers of monomial ideals than with their ordinary powers.  Further, one can use linear algebra to describe Newton polyhedron.  Let $G(I)$ denote the minimal generating system of $I$ consisting of monomials  and
$$d(I) := \max\{|\albf| := \alpha_1 + \cdots + \alpha_r |\ \Xbf^\albf \in G(I) \},$$
the maximal generating degree  of $I$. Let $\ebf_1,...,\ebf_r$ be the canonical basis of $\Rset^r$. In order to avoid trivial cases, we can assume that $r\ge 2$ and $d(I) \ge 2$. 

The following result gives a geometric description of $NP(I)$. It is a variation of  \cite[Lemma 2.2]{HT3}:
 
\begin{lem} \label{NPH} The Newton polyhedron $NP(I)$ is the set of solutions of a system of inequalities of the form
$$\{\xbf\in\Rset^r \mid \abf_i  \xbf  \ge  b_i,\  i=1,\ldots, s\},$$
where $\abf_i=(a_{i1},\ldots,a_{ir})$ are row vectors, such that each hyperplane with the equation $\abf_i  \xbf  = b_i$ defines a facet of $NP(I)$, which contains $t_i$ affinely
independent points of $E(G(I))$ and is parallel to $r - t_i$ vectors of the canonical basis. Furthermore, we can
choose $\zerobf \ne \abf_i  \in \Nset^r$, $ b_i  $  positive integers for all $i = 1, . . . , s$, such that  $a_{i j},\ b_i \le d(I)^{t_i}$ for all $i,\ j$, 
where $t_i$ is the number of non-zero coordinates of $\abf_i$.
\end{lem}

\begin{proof} The first part of the lemma is \cite[Lemma 6]{Tr1}.  For the second part, we need the way to define  $\abf_i$ and $b_i$ given in the proof of  \cite[Lemma 2.2]{HT3}. In order to make the presentation self-contained, let us recall it here.

Let $H$ be a hyperplane which defines the $i$-th facet of $NP(I)$.  W.l.o.g., we may assume that $H$ is defined by $t := t_i$ affinely independent points $\albf_1,...,\albf_t \in E(G(I))$ and is parallel to $r-t$ vectors  $\ebf_{t+1},\ldots,\ebf_r$.  Then the defining equation of $H$ can be written as
$$
\left | \begin{array}{cccc}
x_1& \cdots & x_t & 1\\
\alpha_{11}& \cdots & \alpha_{1t} & 1\\
\vdots & \vdots & \vdots & \vdots\\
\alpha_{t1}& \cdots & \alpha_{tt} & 1
\end{array}
\right | = 0.
$$
Expanding this determinant in the first row, we get: $a'_1x_1+\cdots+ a'_tx_t = b'$, where $a'_i$ are the $(1,i)$-cofactor for $i=1,\ldots,t$ and $b'$ is the $(1,s+1)$-cofactor of this determinant. Then one can  take $a_{ij} = |a'_j|$ for $j\le t$, $a_{ij} = 0$ for $t+1\le j\le r$  and $b_i= |b'|$.
Let $(c_{ij})$ be a square submatrix of rank $t'\le t$  in the above sum. By Hadamard's inequality, we have
$$(\det (c_{ij}))^2 \leqslant \Pi_{i=1}^{t' }(\Sigma_{j=1}^{t'}|c_{ij}|^2) \leqslant   \Pi_{i=1}^{t'}(\Sigma_{j=1}^{t'}|c_{ij}|)^2 \leqslant d(I)^{2t'}.$$ 
From that we immediately get 
the bounds on $a_{i j}$ and $b_i$.
\end{proof}

Further, we need a formula for computing local cohomology modules $H^i_\mfr(R/I)$ given by Takayama in \cite{Ta}. Note that $H_{\mfr}^i(R/I)$ admits an $\Zset^r$-grading over $R$.  For every degree $\albf\in\Zset^r$  we denote by $H_{\mfr}^i(R/I)_{\albf}$ the $\albf$-component of $H_{\mfr}^i(R/I)$.

Recall that a simplicial complex $\Delta $ on a finite vertex set $V$ is a collection of subsets of $V$ such that  $F \subset G$ and $G\in \Delta $ implies $F\in \Delta $ and $\{v\} \in \Delta $ for all $v\in V$. In this case we also write $V= V(\Delta )$. An element of $\Delta $ is called a face, and a maximal element of $\Delta $ is called  a facet. The set of  facets of $\Delta $ is denoted by  $\Fcal(\Delta )$. Clearly, $\Delta $ is uniquely defined by $\Fcal(\Delta )$. In this case we also write $ \Delta = \left<\Fcal(\Delta ) \right>$.

Let $\Delta(I)$ denote the simplicial complex corresponding to the radical ideal $\sqrt I$, i.e., 
$$\Delta (I) := \{ \{i_1,...,i_j\} \subseteq [r]|\ X_{i_1}\cdots X_{i_j} \not\in \sqrt I\} ,$$
where $[r]:= \{1,2,...,r\}$.  For every $\albf = (\alpha_1,\ldots,\alpha_r)^T \in \Zset^r$, we define the negative part of its support  to be the set 
$$\Ga := \{i \ | \ \alpha_i < 0\}.$$  For a subset $F$ of $[r]$, let $R_F := R[X_i^{-1} \ | \ i \in F]$.  
Set
\begin{equation} \label{EQ01}  \Da(I) := \{ F \subseteq [r]\setminus \Ga|\  \Xbf^\albf \notin IR_{F\cup \Ga} \}. \end{equation}
For an example, consider the case $\albf =  \zerobf$.  Then $\supp^{-}(\zerobf) = \emptyset $, and
$$\Delta_{\zerobf}(I) = \{F := \{i_1,...,i_j\} \subseteq [r]|\ 1 \notin IR_F \} = \{ \{i_1,...,i_j\} \subseteq [r]|\ X_{i_1}\cdots X_{i_j} \not\in \sqrt I\} =  \Delta (I). $$
This was mentioned in  \cite[Example 1.4(1)]{MT}.

We set $\widetilde{H}_i(\emptyset ; K) = 0$ for all $i$,   $\widetilde{H}_i( \{\emptyset \} ; K) = 0$ for all $i\neq -1$, and $\widetilde{H}_{-1} (\{ \emptyset\} ; K) =  K$. Using \cite[Lemma  1.1]{GH} we may write Takayama's formula as follows.

 \begin{lem}\label{Takay} {\rm (\cite[Theorem 2.2]{Ta})}  We have
 $$ \dim_K H_{\mfr}^i(R/I)_{\albf } = \left\{ \begin{array}{ll}
\dim_K \widetilde{H}_{i-|\Ga|-1}(\Da(I); K)  & \text{if } \  \Ga \in\Delta(I),\\
0       & \text{otherwise}.
\end{array}
\right.
$$
\end{lem} 
Let $\Gamma \subseteq \Delta (I)$ be a simplicial subcomplex with the vertex set $V(\Gamma ) \subseteq [r]$.
We set
\begin{equation} \label{ETa1}
a_{\Gamma, i} (I) = \begin{cases} \sup\{ |\albf|\  |\ \albf\in \Zset^r \ \text{and} \ \Da(I) = \Gamma \} \ & \text{if} \  \widetilde{H}_{i + | V(\Gamma ) | - r -1}(\Gamma ; K) \neq 0 \}, \\
-\infty  & \text{otherwise}.
\end{cases}
\end{equation}
Then we get the following immediate consequence of Lemma \ref{Takay} and of the Artinian property of local cohomology modules:

\begin{cor} \label{TakayB} $a_i (R/I) = \max \{ a_{\Gamma ,i}(I) |\ \Gamma \subseteq \Delta (I) \ \text{is a simplicial subcomplex}\}$.  Moreover, if  $\widetilde{H}_{i + | V(\Gamma ) | - r -1}(\Gamma ; K) \neq 0 \}$, then $a_i (R/I) \in \Zset$.
\end{cor}

 Recall that the support of a vector $\abf \in \Rset^r$ is
$$\supp(\abf) =\{i| \ a_i\neq 0\}.$$
Using (\ref{EN2}), Lemma \ref{NPH} and the definition (\ref{EQ01}), one can describe $\Da(\Inba)$ via  a system of linear constraints.  The case $\albf \in \Nset^r$ is given in \cite[Lemma 3.1]{HT3}. In the general case, $\albf \in \Zset^r$, we already have

\begin{lem} {\rm (\cite[Lemma 1.2]{HT2})} \label{FNPn2}  Assume that $\Ga \in \Delta (I) $ for some $\alpha \in \Zset^r$. Then 
$$\Da(\overline{I^n}) =\{ F\setminus \Ga \mid \Ga \subseteq F,\ F \in \Delta (I), \ \text{and}\ \Xbf^\albf \not\in \Inba R_F \}.$$
\end{lem}
Now we can give a generalization of  \cite[Lemma 3.1]{HT3}.

\begin{lem} \label{FNPn} Keep the notations in Lemma \ref{NPH}. Let $\albf \in \Zset^r,\ n\ge  1$. Assume that $\Ga \in \Delta (I) $. Then we have
$$\begin{array}{ll} \Da(\overline{I^n}) =\langle [r]\setminus ( \supp(\abf_i) \cup \Ga) \mid & i\le s; \  \Ga \subseteq  [r]\setminus  \supp(\abf_i) \\
& \text{ and } \sum_{j \not\in \Ga} a_{ij} \alpha_j <nb_i \rangle.
\end{array}$$
\end{lem}

\begin{proof}  The following proof is an adaptation  of that of \cite[Lemma 3.1]{HT3}  to our case. In order to make the presentation self-contained, let us give the detail.

Let $F\in \Da(\overline{I^n})$. By Lemma \ref{FNPn2}, $F= F' \setminus \Ga$ for some $F' \in \Delta (I)$ such that $\Ga \subseteq F'$ and $\Xbf^\albf \not\in \Inba R_{F'}$. We may assume that $F' =\{p,\ldots,r\}$ for some $1\le p \le r+1$ (where $p=r+1$ means that $F' = \emptyset $).  Note that  $\Xbf^{\albf} \notin \overline{I^n}R_{F' }$ if and only if $\Xbf^{\albf} \Xbf^{\gabf} \notin \overline{I^n}$ for any monomial $\Xbf^{\gabf}\in K[X_{p},\ldots, X_r]$. Fix a monomial  $\Xbf^{\gabf} = X_p^m\cdots X_r^m$ with $m\gg 0$. By Lemma \ref{NPH}, (\ref{EN1})  and (\ref{EN2}),   there is $1\le i \le s$ such that
$$ \abf_i \albf + m(a_{i p} + \cdots  + a_{ir}) = \abf_i \cdot [\albf+ m( \ebf_p + \cdots + \ebf_r)] < nb_i .$$
Recall, by Lemma \ref{NPH}, that $a_{ij} \ge 0$ for all $j$.  Hence we must have $a_{i p} = \cdots = a_{ir} = 0$, whence $F' \subseteq [r] \setminus \supp(\abf_i)$, and
$$\abf_i \albf =  \abf_i [\albf+ m( \ebf_p + \cdots + \ebf_r) ] < nb_i.$$
Since $\Ga \subseteq F' \subseteq [r] \setminus \supp(\abf_i) = \{j \mid \ a_{ij} =0\}$, we have 
$$\sum_{j \not\in \Ga} a_{ij} \alpha_j  = \sum_{1\le j \le r} a_{ij} \alpha_j  = \abf_i \albf < nb_i.$$
From the equality $F= F' \setminus \Ga$ and the inclusion $F' \subseteq [r] \setminus \supp(\abf_i) $ we also get $F \subseteq [r] \setminus ( \supp(\abf_i) \cup \Ga)$.

Conversely, let $F \subseteq [r] \setminus ( \supp(\abf_i) \cup \Ga)$ for some $1\le i \le s$ such that $\Ga \subseteq [r] \setminus \supp(\abf_i)$ and $\sum_{j \not\in \Ga} a_{ij} \alpha_j  < nb_i $. Note that $F':= F \cup \Ga \subseteq [r] \setminus  \supp(\abf_i)$. W.l.o.g., we may assume that $F'= \{p,...,r\},\ 1\le p \le r+1$. Then $a_{ij} = 0$ for all $p \le j \le r$, and for any $X^{\gamma_p}\cdots X^{\gamma_r} \in K[X_p,...,X_r]$, we have 
$$\abf_i ( \albf + (0,...,0,\gamma_p , ..., \gamma_r)^T)  = \sum_{j \not\in \Ga} a_{ij} \alpha_j  < nb_i .$$
  Using (\ref{EN1}) and (\ref{EN2}) and Lemma \ref{NPH}, we get that $\Xbf^{\albf} X^{\gamma_p}\cdots X^{\gamma_r} \not\in \Inba$, or  equivalently, $\Xbf^{\albf}  \not\in \Inba R_{F'}$. By Lemma \ref{FNPn2}, $F\in \Da(\Inba)$.

Now using the maximality of  facets, we immediately get the statement of the lemma.
\end{proof}

Let  $\Gamma $ be  a simplicial subcomplex  of $\Delta (I)$ such that $\Gamma = \Da(\Inba)$ for some $\albf \in \Zset^r$ and $n\ge 1$. W.l.o.g., we may assume that $V(\Gamma ) = [r']$, where $r'\le r$, so that $\Ga =\{ r'+1,...,r\}$. Further,  assume that $\Ga \subseteq [r]\setminus \supp(\abf_i)$ for $i=1,..., \stil$ and $\Ga \not\subseteq [r]\setminus \supp(\abf_i)$ for $i>\stil$, where $\stil \le s$. Then, by Lemma \ref{FNPn}, we may further assume  that
\begin{equation} \label{ECM10} \Gamma = \left< [r]\setminus ( \supp(\abf_i) \cup \Ga) |\ i = 1,...,s' \right>,
\end{equation}
where $s' \le \stil$.
For a row vector $\abf  = (a_1,...,a_r) \in \Rset^r$,  denote $\abf' := (a_1,...,a_{r'}) \in \Rset^{r'}$. Similarly for column vectors. Consider the following polyhedron
\begin{eqnarray}\label{ECM11}
Q_{\Gamma , n}:=\{\xbf' \in\Rset^{r'} |\    \abf'_i  \xbf'   & \le & nb_i - 1 \  (i=1,...,s'),\\
  \abf'_l  \xbf'  & \ge &  nb_l \  (l = s'+1,...,\stil), \nonumber \\
  \text{and}\ x_j  &\ge& 0 \ (j=1,..., r') \}. \nonumber
\end{eqnarray}

Together with Corollary \ref{TakayB}, the following result translates the problem of computing the $a_i$-invariant $a_i(R/\Inba)$  into solving a finite set of integer programming problems.

\begin{thm} \label{CM7} Assume that $\widetilde{H}_{i + r' - r -1}(\Gamma ; K) \neq 0$. Then
$$a_{\Gamma , i}(\Inba) = \sup \{ x_1 + \cdots + x_{r'}|\ \xbf' \in Q_{\Gamma , n} \cap \Nset^{r'} \}+ r' - r.$$
\end{thm}

\begin{proof}   Let
$$L : = \{  \albf \in \Zset^r \  |\ \Da(I^n) = \Gamma \}.$$
By Lemma \ref{FNPn} and from (\ref{ECM10}) we must have $\abf'_l  \albf'   \ge   nb_l $ for any $l = s'+1,...,\stil$ and 
$\abf'_i  \albf'    <  nb_i $ for any $i=1,...,s'$.  Since $\abf_i , \ \albf $ are integer    vectors and $b_i$ is an integer, the last inequality is equivalent to $\abf'_i  \albf'   \le  nb_i - 1 $. This means $\albf' \in Q_{\Gamma , n}$, provided that $\albf \in L$. Note that the condition $\Da(\Inba) = \Gamma$ implies $\Ga = \{r'+1,...,r\}$, so that $\alpha_j \le -1$ for all $j= r'+1,...,r$. Hence
$$\sup \{ \alpha_1 + \cdots + \alpha_r |\ \albf \in L \} \le  \sup \{ \alpha_1 + \cdots + \alpha_{r'}  |\ \albf' \in Q_{\Gamma , n} \cap \Nset^{r'}\} - (r-r').$$
On the other hand, if $(\beta_1,...,\beta_{r'})^T \in Q_{\Gamma , n} \cap \Nset^{r'}$, then $(\beta_1,...,\beta_{r'}, -1, ..., -1)^T \in L$ (the number of $-1$'s is $r-r'$). So, the above inequality is in fact an equality. Now the lemma follows from (\ref{ETa1}).
\end{proof}
 
 In this paper, we do not intend to compute the $a_i$-invariant. We only want to study the asymptotic behavior of $a_i(R/\Inba)$. Since there are only finitely many simplicial subcomplexes of $\Delta (I)$, we can restrict our problem to a fixed $\Gamma $, which satisfies the assumption of  Theorem \ref{CM7}.  Moreover, we do not need to compute any reduced homology group of $\Gamma $. 
 
From now on, fix $\Gamma $ as above.  Consider the following polyhedron:
\begin{eqnarray}\label{ECM71}
P_{\Gamma , n}:=\{\xbf' \in\Rset^{r'} |\    \abf'_i  \xbf'   & \le & nb_i  \  (i=1,...,s'),\\
  \abf'_l  \xbf'  & \ge &  nb_l \  (l = s'+1,...,\stil), \nonumber \\
  \text{and}\ x_j  &\ge& 0 \ (j=1,...,r') \}. \nonumber
\end{eqnarray}
From (\ref{ECM11}) it implies that $Q_{\Gamma , n} \subseteq P_{\Gamma , n}$.  In order to apply Theorem \ref{I9} to the integer program in Theorem \ref{CM7}, we need the  following result. This  was shown in the proof of \cite[Lemma 3.2]{HT3}. For the convenience of the reader, we give a  proof here.

\begin{lem} \label{CM8} Assume that $\widetilde{H}_{i + r' - r -1}(\Gamma ; K) \neq 0$ and $Q_{\Gamma , n} \cap \Nset^{r'} \neq \emptyset  $ for some $n\ge 1$. Then   $P_{\Gamma , 1}$ is a bounded and full-dimensional polyhedron in $\Rset^{r'}$.
\end{lem}

\begin{proof}   Since $\widetilde{H}_{i + r' - r -1}(\Gamma ; K) \neq 0$, there is   $\albf \in Q_{\Gamma , n} \cap \Nset^{r'}$.  Then $\albf / n \in P_{\Gamma , 1}$, whence $P_{\Gamma , 1} \neq \emptyset $.  Recall by Lemma \ref{NPH} that $a_{ij} \ge 0$. If $a_{ij} = 0$ for some $j\le r'$ and all $i=1,...,s'$, then $\albf + m\ebf_j \in Q_{\Gamma , n} \cap \Nset^{r'}$ for all $m\ge 0$,  where by abuse of notation, $\ebf_1, ..., \ebf_{r'}$ also  denote the canonical basis of $\Rset^{r'}$.  Hence  $a_{\Gamma ,i}(R/ \Inba) = \infty $, a contradiction (see Corollary \ref{TakayB}) . From this it follows  that $s'\ge 1$, and $\sum_{1\le i\le s'}a_{ij} >0$ for all $j\le r'$.  For all $\xbf' \in P_{\Gamma , 1}$ we have $\sum_{1\le j \le r'}(\sum_{1\le i\le s'}a_{ij}) x_j  \le \sum_{1\le i \le s'} b_i$,  which yields that $P_{\Gamma , 1}$ is  bounded.

Further, let $a^* = \max_{i,j} |a'_{ij}|$. Then for all $0\le \varepsilon_1,...,\varepsilon_{r'} \le 1/(2r'a^*)$ and any $i\le s'$,  we have 
$$\abf'_i\cdot (\albf + \varepsilon_1 \ebf_1 + \cdots + \varepsilon_{r'}\ebf_{r'}) \le nb_i - 1 + 1/2 < nb_i.$$
 By (\ref{ECM71}) it implies that $\albf + \varepsilon_1 \ebf_1 + \cdots + \varepsilon_{r'}\ebf_{r'} \in P_{\Gamma ,n} = nP_{\Gamma ,1}$. Hence  $\albf/n + (\varepsilon_1/n) \ebf_1 + \cdots + (\varepsilon_{r'}/n) \ebf_{r'} \in  P_{\Gamma , 1}$, which yields that $P_{\Gamma , 1}$ is a  full-dimensional polyhedron in $\Rset^{r'}$.
\end{proof}

In order to apply Theorem \ref{I9} to  the integer program in Theorem \ref{CM7}, in the sequel let us consider the following set up
\begin{eqnarray} \label{ECM81}  \At  & = & \begin{pmatrix} a_{11} & a_{12} & \cdots & a_{1 r'}\\
& \cdots   &  \cdots & \\
a_{s'1} & a_{s'2} & \cdots & a_{s' r'}\\
- a_{(s'+1)1} & - a_{(s'+1)2} & \cdots & - a_{(s'+1) r'}\\
& \cdots   &  \cdots & \\
- a_{\stil 1} & - a_{\stil 2} & \cdots & - a_{\stil  r'}
\end{pmatrix},  \\  \nonumber
\widetilde{\bbf} &=& \begin{pmatrix} b_1\\ \vdots \\ b_{s'}\\  - b_{s'+1}\\ \vdots \\ -b_{\stil} \end{pmatrix}, \  \ 
\cT = \begin{pmatrix} -1\\ \vdots \\ -1\\ 0 \\ \vdots \\ 0 \end{pmatrix},  \ \  \dT=  \begin{pmatrix} 1\\ \vdots \\ 1 \end{pmatrix}, 
\end{eqnarray}
where the number of $-1$ in $\cT\in \Rset^{r'}$ is $s'$,  $a_{ij}, \ b_i$ are defined in Lemma \ref{NPH}, and $\dT $ is an $r'$-vector. (In the rows starting from $(s'+1)$-st in $\At$ and $\widetilde{\bbf}$ above we have to put  the minus sign in order to keep the type of constraints considered in Sections \ref{Opt} and \ref{IOpt}.)

From Theorem \ref{CM7}, Corollary \ref{TakayB} and Theorem \ref{I9} we immediately get that each $a_i(R/\Inba)$ is a quasi-linear function for $n\ge 0$. This fact was proved in  \cite[Theorem 4.1]{HT1}  by a different method. However, our main aim is to provide a bound, from where $a_i(R/\Inba)$ becomes a quasi-linear function. In order to do that, we first determine a number $N_*$ such that $a_{\Gamma , i} (R/\Inba)$ is a quasi-linear function for all $n\ge N_*$. Moreover, some additional information on the coefficients of linear functions of $a_{\Gamma , i} (R/\Inba)$ are also  given.

\begin{lem} \label{CM9} In the setting of (\ref{ECM81}), let $D, \ D'$ and $\Delta $ be  the maximum absolute value of the subdeterminants of the matrices $\At$, $[\At\ \bT]$ and $[\At\ \bT\ \cT]$, respectively. Set 
$$N_*:= (r+1)D'(D-1) + (r+2)\Delta - D .$$
 Assume that $a_{\Gamma ,i}(\overline{I^{n_0}})$ is finite for some $n_0 \ge 1$. Set 
$$\psi := \max \{x_1 + \cdots + x_{r'} \mid \xbf' \in P_{\Gamma , 1}\}.$$
Then there are linear functions $f_0,...,f_{t-1}$ of the form $f_j(n) = \psi n + \beta_j$, where $t\ge 1$, such that
$$a_{\Gamma ,i}(R/\Inba) = f_j(n) \ \ \text{whenever }\ \ n \equiv j \nmod t$$
for all $ n \ge N_*$.

Moreover, $\psi $ is a non-negative rational number with positive  denominator at most $D $, and
$$-2 r^2D \le \beta_j \le 0 \ \ (0\le j \le t-1).$$
\end{lem} 

\begin{proof}  Since $a_{\Gamma ,i}(\overline{I^{n_0}})$ is finite, by Theorem \ref{CM7},  $\widetilde{H}_{i + r' - r -1}(\Gamma ; K) \neq 0$ and $Q_{\Gamma , n_0} \cap \Nset^{r'} \neq \emptyset  $. By Lemma \ref{CM8},  $P_{\Gamma , 1}$ is a bounded and full-dimensional polyhedron in $\Rset^{r'}$.  Hence $\psi $ is finite. Applying Theorem \ref{I9} to the setting (\ref{ECM81}), we then get the existence of  $f_j(n) = \psi n + \beta_j$, $0\le j\le t-1$, such that
$$a_{\Gamma ,i}(R/\Inba) = f_j(n) \ \ \text{whenever }\ \ n \equiv j \nmod t$$
for all $ n \ge N_*$.

Let 
$$\begin{array} {ll}
\psi_n &:= \max\{ x_1 + \cdots + x_{r'} \mid \xbf' \in P_{\Gamma , n} \}, \\
\Psi_n & := \max\{ x_1 + \cdots + x_{r'} \mid \xbf' \in Q_{\Gamma , n} \}, \\
M_n & := \max\{ x_1 + \cdots + x_{r'} \mid \xbf' \in Q_{\Gamma , n} \cap \Nset^{r'} \}.
\end{array}$$
Note that $\psi = \psi_1$  reaches its value at a vertex of $P_{\Gamma , 1}$. From that all properties of $\psi $ follow.

It is left to give bounds on $\beta_j$.  By Theorem \ref{CM7}, $a_{\Gamma ,i} (R/\Inba) = M_n + r'-r$. Note that $\psi_n =  \psi n$. Since $Q_{\Gamma , n} \subseteq P_{\Gamma , n}$, we have $M_n \le \Psi_n \le \psi_n = \psi n$. Hence $\beta_j \le r' - r \le 0$ for all $j\le t-1$.

The idea to show the lower bound on $\beta_j$ is the following:  Since $\Psi_n$ is attained at a vertex of $Q_{\Gamma , n}$, we can compute an optimal solution for it. Well-known results in Integer Programming say that this solution cannot be too far from an optimal integer   solution for $M_n$. 

Indeed, by \cite[Theorem 17.2]{Sch}, there is an optimal solution $\vbf$ for $\Psi_n$ and an optimal solution $\ubf$  for $M_n$ such that $|v_j - u_j|  \le rD$ for all $j\le r$. Note that $\dT^T \vbf = \Psi_n$. Hence
$$ M_n = \dT^T \ubf = \dT^T \vbf + \dT^T (\ubf - \vbf) = \Psi_n + \dT^T (\ubf - \vbf) \ge \Psi_n - r^2 D.$$
 On the other hand, $\Psi_n$ reaches its value at a vertex of $Q_{\Gamma , n}$.  This vertex has the form
$\vbf' = n \zbf_{\Ibf, \bT} + \zbf_{\Ibf, \cT}$ for some $ I \in \Ical $ (see (\ref{EO24}) and (\ref{EO25})).  Hence 
$$\Psi_n = n [(\zbf_{\Ibf,\bT})_1 + \cdots + (\zbf_{\Ibf,\bT})_{r'}] + \beta',$$
where $\beta' = (\zbf_{\Ibf,\cT})_1 + \cdots + (\zbf_{\Ibf,\cT})_{r'}$. By Proposition \ref{O3}, we must have 
$\Psi_n = \psi n + \beta'.$
By Cramer's rule, $|(\zbf_{\Ibf,\cT})_j| \le |D_j|$, where $D_j$ is the determinant of a matrix obtained from $\At (i_1,...,i_k; j_1,...,j_k)$ by replacing the $j$-th column by $(\ct_{i_1},..., \ct_{i_k})^T$. Expanding $D_j$ by this column, we see that $| (\zbf_{\Ibf,\cT})_j| \le k D \le r' D$. Hence $\beta' \ge - ( r')^2D$, which yields 
$$ a_{\Gamma ,i} (R/\Inba)  = M_n + r'-r  \ge \psi n -  (r')^2 D - r^2 D + r' - r  \ge \psi n - 2 r^2D.$$
This implies $\beta_j \ge -2r^2D$, as required.
\end{proof}

We are now ready to state the main result on $a_i$-invariants.

\begin{thm} \label{CM10}  In the setting of (\ref{ECM81}),  let $D, \ D'$ and $\Delta $ be  the maximum absolute value of the subdeterminants of the matrices $\At$, $[\At\ \bT]$ and $[\At\ \bT\ \cT]$, respectively. Then

(i) The invariant $a_i(R/\Inba)$ is a quasi-linear function of the same slope for all
$$n\ge \max\{ (r+1)D'(D-1) + (r+2)\Delta  - D,\ 2r^2D^3\}.$$

(ii) In particular, $a_i(R/\Inba)$ is a quasi-linear function of the same slope for all $n\ge N_\dagger := 2r^{2+3r/2}d(I)^{3r^2}$.
\end{thm}

\begin{proof}  (i) The intuition of the proof is quite simple. By Lemma \ref{CM9}, each finite $a_{\Gamma ,i}(R/\Inba)$ is a quasi-linear function of $n$ for all $n\gg 0$. By choosing the product of  periods of all  such $a_{\Gamma ,i}(R/\Inba)$ as a new period $\tau $ of all $a_{\Gamma ,i}(R/\Inba)$, and hence also a period   of $a_i(R/\Inba)$, by Corollary \ref{TakayB}, we see that in each equivalence class modulo $\tau $, $a_i(R/\Inba)$ is the maximum of a finitely many linear functions. As already mentioned at  the end of the proof of Proposition \ref{O3},  after the largest coordinate of intersection points of corresponding lines, there is one line lying above all other ones. This is the graph of $a_i(R/\Inba)$.

For the detail, set
$$N_*:= (r+1)D'(D-1) + (r+2)\Delta  - D.$$
Let $n\ge  \max\{ N_*,\ 2r^2D^3\}$ be any integer. By Corollary \ref{TakayB},
\begin{equation} \label{ECM101} 
a_i (R/\Inba) = \max \{ a_{\Gamma ,i}(R/\Inba) |\ \Gamma \subseteq \Delta (I) \ \text{is a simplicial subcomplex}\}.
\end{equation}
Assume that $\Gamma_1$ and $\Gamma_2$ are two simplicial subcomplexes in the right hand side of (\ref{ECM101}) such that both $a_{\Gamma_1 ,i}(R/\Inba)$ and $a_{\Gamma_2 ,i}(R/\Inba)$ are finite. By Lemma \ref{CM9},
$$\begin{array}{ll}
a_{\Gamma_1 ,i}(R/\Inba) &= \alpha_1 n + \beta_1(n),\\
a_{\Gamma_2 ,i}(R/ \Inba) &= \alpha_2 n + \beta_2(n),
\end{array}$$
where $\beta_1(n)$ and $\beta_2(n)$ are periodic functions (of the same  period $\tau $). If $\alpha_1  \ge \alpha_2$ and $\beta_1(n) \ge \beta_2(n)$, then we may delete $a_{\Gamma_2 ,i}(R/ \Inba)$ in the right hand  side of (\ref{ECM101}). 

So, we may now assume that $\alpha_1 > \alpha_2$ and $\beta_2(n) > \beta_1(n)$.
Note that 
\begin{equation} \label{ECM102}
\alpha_1 n + \beta_1(n) \ge \alpha_2 n + \beta_2(n) \ \text{if } \ n\ge \frac{\beta_2(n) - \beta_1(n)}{\alpha_1 - \alpha_2}.
\end{equation}
By Lemma \ref{CM9} one can write $\alpha_i = \frac{\alpha_{i1}}{\alpha_{i2}}$,  ($i=1,2$), such that $\alpha_{i1}$ is an integer, and $\alpha_{i2}\le D$ is a positive integer. Then
$$\alpha_1 - \alpha_2 = \frac{\alpha_{22} \alpha_{11} - \alpha_{12}\alpha_{21}}{\alpha_{12}\alpha_{22}} \ge \frac{1}{D^2}.$$
On the other hand, again by Lemma \ref{CM9}, $\beta_2(n) - \beta_1 (n)\le 2r^2D$. Hence
$$\frac{\beta_2(n) - \beta_1(n)}{\alpha_1 - \alpha_2} \le 2r^2D^3 .$$
Since $n\ge \max\{ N_*,\ 2r^2D^3\}$, from (\ref{ECM102}) we get  $a_{\Gamma_1 ,i}(R/\Inba) \ge a_{\Gamma_2 ,i}(R/\Inba)$, so that we can also delete $a_{\Gamma_2 ,i}(R/\Inba)$ in the right hand side of (\ref{ECM101}). In both cases, the bigger slope is the slope of the maximum of  two quasi-linear functions. Using this fact, we can conclude that for each fixed $j\le \tau $ we can find  a subcomplex $\Delta_j$ of $\Delta (I)$ such that $a_i(R/\Inba) = a_{\Delta_j,i}(R/ \Inba)$ for all $n\ge \max\{ N_*,\ 2r^2D^3\}$  such that $n\equiv j \nmod \tau $.

Now (i) follows from Lemma \ref{CM9}.
\vskip0.3cm 

(ii) Denote by 
$$b^* := \max\{ |b_1|,...,|b_s|\} \ \text{and}\ a^*:= \max\{|a_{ij}||\ i\le s,\ j\le r\}.$$
 Note that the matrix $[\At \ \bT \  \cT ]$ has a  rank at most $r'+2 \le r+2$. Using  the Laplace expansion along the last column of a determinant of $[\At \ \widetilde{\bbf } \  \cT ]$  we can conclude that $\Delta \le (r+2)D'$. Similarly, 
$D' \le (r+1)b^*D$.  Then 
$$N_* \le (r+1)^2 b^* D^2 + (r+2)^2(r+1)b^* D.$$
Applying Hadamard's  inequality to a subdeterminant of $\At$ and using Lemma \ref{NPH}, there is a $(t\times t)$-submatrix of $\At$ such that
$$D^2 \le (\tilde{a}^2_{i_1j_1} + \cdots + \tilde{a}^2_{i_1j_t}) \cdots (\tilde{a}^2_{i_tj_1} + \cdots + \tilde{a}^2_{i_tj_t})  \le [r (a^*)^2  ]^r \le r^r d(I)^{2r^2},$$ 
whence  $D\le r^{r/2} d(I)^{r^2}$.  Since $b^* \le d(I)^r$ by Lemma \ref{NPH}, it implies
$$\begin{array}{ll} N_* & < (r+1)^2 d(I)^r  r^r d(I)^{2r^2} + (r+2)^2 (r+1) r^{r/2} d(I)^r d(I)^{r^2}\\
&= d(I)^r  r^r d(I)^{2r^2} (r+1) [(r+1) + \frac{(r+2)^2}{r^{r/2} d(I)^{r^2} }]\\
&< d(I)^r  r^r d(I)^{2r^2}(r+1) [r+1 +1] \ \ \text{(since }\ r,d(I)\ge 2 \text{)},
\end{array}$$
whence 
\begin{equation} \label{ECM103}
N_* < (r+1)(r+2)r^r d(I)^{2r^2} <  2r^{2+3r/2}d(I)^{3r^2} .
\end{equation}
We also have $2r^2D^3  \le  2r^{2+3r/2}d(I)^{3r^2}.$ Therefore  (ii) now follows from (i).
\end{proof}

From the definition (\ref{Ereg2}) and Theorem \ref{CM10} we can only conclude that $\reg(\Inba)$ is a quasi-linear function for $n\gg 0$. However, more than twenty years ago, using totally different methods,  Cutkosky-Herzog-Trung \cite{CHT} and independently Kodiyalam \cite{Ko} proved that $\reg(R/\Inba)$ is a linear function of $n$ for $n\gg 0$.  Moreover, if $I$ is a monomial ideal, the behavior of $\reg(\Inba)$ is much more understood, as the following result shows.

\begin{lem} {\rm (\cite[Theorem 4.10]{HT1}) }\label{CM11a} Let $I$ be a non-zero monomial ideal of $R$. Then there are a positive integer $p$ and  a nonnegative integer $0\le e\leq \dim (R/I)$ such that
$\reg(\overline{I^n}) = pn+e$ for all $n\gg 0$. Moreover $pn \le \reg(\overline{I^n}) \leq  pn + \dim(R/I)$ for all $n>0$. 
\end{lem}

In the following  result we can give an upper bound for the stability index of $\reg(\Inba)$. 

\begin{thm} \label{CM11}
Let $I$ be a non-zero monomial ideal of $R= K[X_1,...,X_r]$ of maximal generating degree $d(I)$. Then there are a positive integer $p\le d(I)$ and a non-negative integer $0\le e \le \dim R/I$ such that $\reg(\Inba) = pn +e$ for all 
$$n\ge (r+1)(r+2)r^r d(I)^{2r^2},$$
and $pn \le \reg(\Inba) \le pn + \dim(R/I)$ for all $n\ge 1$. In particular
$$\regb (I) \le (r+1)(r+2)r^r d(I)^{2r^2}.$$
\end{thm}

\begin{proof} The existence of $p$ and $0\le e \le \dim (R/I)$, such that $\reg(\Inba) = pn +e$ for all  $n\gg 0$ and $pn \le \reg(\Inba) \le pn + \dim(R/I)$ for all $n\ge 1$ is the content of  Lemma \ref{CM11a}.

It is left to  prove the lower bound for $\regb (I)$.  The main idea is to find in each equivalence class of $n$ modulo a period $t$ an index $j$ and a simplicial complex  $\Gamma_0$  such that 
$$a_{\Gamma_0 ,j}(R/\Inba) + j =  \reg(R/\Inba) =pn + e-1,$$
for all $n\gg 0$. 

Let $N$ be  a number  such that 
$$\reg(\Inba) = pn + e, \ \text{for all } n\ge N.$$
Keep the notation in Theorem \ref{CM10}. Let
$$N_*:= (r+1)D'(D-1) + (r+2)\Delta - D .$$
 By enlarging  periods we can assume that all eventually quasi-linear functions $a_{\Gamma ,i}(R/\Inba)$ have the same period $t$ for some $t\ge 1$. Since 
$\reg(\Inba) = 1 + \max \{a_i(R/\Inba) + i \mid 0\le i\le \dim R/I \}$, from Corollary \ref{TakayB} and Lemma \ref{CM9} it implies that if $a_{\Gamma ,i}(R/\Inba)$ is finite, then for all $n\ge N_*$, 
\begin{eqnarray} \label{ECM111} a_{\Gamma ,i}(R/\Inba)  &= &\alpha_{\Gamma,i}n + \beta_{\Gamma , i; k}\ \  \text{if\ } n \equiv k \nmod t, \\  \nonumber
\text{such that  }  \alpha_{\Gamma,i} &\le& p, \ \text{and if } \alpha_{\Gamma,i}n = p\  \text{ then  }\beta_{\Gamma , i;k} + i +1 \le e.
\end{eqnarray}
Fix an integer $m\ge N_*$.  Let $n_0= m + t N$. Again by Corollary \ref{TakayB}, there are $0\le j\le \dim R/I $  and a simplicial subcomplex $\Gamma_0 $ such that
$$\reg(R/\overline{I^{n_0}})  = a_{\Gamma_0 ,j}(R/\overline{I^{n_0}}) + j .$$
By the choice of $t$ and $N$, it implies
$$a_{\Gamma_0 ,j}(R/\overline{I^{n_0}}) + j  = pn_0 + e.$$
 On the other hand, by Lemma \ref{CM9},  there exist  a non-negative rational number $ \alpha \le p$ with positive  denominator at most $D$   and $\beta \le 0$, such that
$$a_{\Gamma_0 ,j}(R/\Inba) = \alpha n + \beta ,$$
for all $n\ge N_*$ and $n \equiv m (\nmod t)$. Hence
$$pn_0 + e = \alpha n_0 + \beta + j + 1.$$
   If $\alpha < p$, then $p-\alpha \ge 1/D$, whence
  $$r+1 \ge \beta + j +1 - e = (p-\alpha )n_0 \ge (N_* + tN)/D > N_*/D \ge r+1,$$
 a contradiction. So $\alpha = p$ and $\beta + j+ 1 = e$.  This means  
$$a_{\Gamma_0 ,j}(R/\Inba) = pn + e - j -1 ,$$
for all $n\ge N_*$ and $n \equiv m (\nmod t)$. In particular
$a_{\Gamma_0 ,j}(R/\Imba) + j+ 1 =   pm +e.$
Combining with (\ref{ECM111}) we can conclude that $\reg(\Imba) = pm+e$. Since $m$ can be any number bigger or equal $N_*$, we get that $\reg(\Inba) = pn+e$ for all $n\ge N_*$.  Using the estimation of $N_*$ in (\ref{ECM103}), we get the desired upper bound for $\regb (I)$.
\end{proof}

\begin{rem} \label{CM12} Fix $r\ge 4$. The monomial ideal $I\subset R$ given in \cite[Proposition 17]{Tr1} has $\depth(R/\Inba) > 0$ for $n\le n_0 = O(d(I)^{r-2})$ and $\depth(R/\Inba) = 0$ for $n\gg 0$. This means $a_0(R/\Inba)  $  are finite  for all $n \le n_0$ and $a_0(R/\Inba) = - \infty $ for $n\gg 0$, i.e.,  $a_0(R/\Inba)$ becomes linear only if $n $ is at least $O(d(I)^{r-2})$.  Unfortunately, this example does not  show that the bound on $\regb(I)$ should also be at least $O(d(I)^{r-2})$.
\end{rem}

Finally, we consider the Castelnuovo-Mumford regularity of the so-called symbolic powers of a  square-free monomial ideal. Recall that a monomial ideal is said to be square-free if its generators are products of different variables. In this case one can write 
$$I = \cap_{i=1}^s \pfr_i ,$$
where $\pfr_i = (X_{t_{i1}},...,X_{t_{im_i}})$ is  the prime ideal generated by a subset of variables (i.e., $\{t_{i1}, ..., t_{im_i}\} \subseteq [r]$). Then the symbolic $n$-th power of $I$ is defined by
$$I^{(n)} = \cap_{i=1}^s \pfr_i^n.$$
Note that 
$$\Delta (I^{(n)}) = \Delta (I) = \left< F_1,...,F_s\right>,\ \ \text{where}\ \ F_i := [r]\setminus \{t_{i1},...,t_{i m_i}\}.$$
For an arbitrary homogeneous ideal $J$, the definition of the symbolic power $J^{(n)}$ is more complicated, and very little is known about $\reg(J^{(n)})$, see, e.g., \cite[Section 2]{HHT}. However, for the case of square-free monomial ideals, as a consequence of \cite[Theorem 4.1 and Theorem 4.7]{HT1},  we get

\begin{cor} \label{Sym1} If $I$ is  a square-free monomial ideal of $R$, then for all $i\le \dim R/I$, $a_i(R/I^{(n)})$ is a quasi-linear function of the same slope for all $n\gg 0$.
\end{cor}

In the case of symbolic powers, the following result plays the role of Lemma \ref{FNPn}.

\begin{lem} {\rm (\cite[Lemma 1.3]{HT2} and also \cite[Lemma 2.1]{MT})} \label{Sym2a} Let $I$ be  a square-free monomial ideal of $R$. Assume that   $\Ga \in \Delta (I)$ for some $\albf\in \Zset^r$. Then for all $n\ge 1$, we have
$$\Da(I^{(n)}) = \left < F\setminus \Ga \mid  F\in \Fcal(\Delta (I)),\  \Ga \subseteq F \ \text{and}\ \sum_{j\notin F} \alpha_j  \le n-1\right>.$$
\end{lem}

Using this lemma, Corollary \ref{TakayB} and (\ref{ETa1}), one can again translate the problem of defining a place from where $a_i(R/I^{(n)})$ becomes a quasi-linear function to the study of the asymptotic behavior of a family of integer programs. From that we get the following result on the index of stability of the Castelnuovo-Mumford regularity $\reg(I^{(n)})$.

\begin{thm} \label{Sym2} Let $I$ be  a square-free monomial ideal of $R$. For any $i\le \dim R/I$, $a_i(R/I^{(n)})$ is a quasi-linear function of the same slope for all $n\ge 2r^{2+3r/2}$. In particular, $\reg(I^{(n)})$  is a quasi-linear function of the same slope for all $n\ge 2r^{2+3r/2}$.
\end{thm}

\begin{proof} The proof of this theorem is similar to that of Theorem \ref{CM10}.  Fix a simplicial subcomplex $\Gamma $ of $\Delta (I)$ such that $\Gamma = \Da(I^{(n)}) $ for some $n\ge 1$ and some $\albf \in \Zset^r$
 with $\Ga \in \Delta (I)$. W.l.o.g., we may assume that $\Ga =\{ r'+1,...,r\}$, i.e.,  $V(\Gamma ) = [r']$, where $r'\le r$, and that $\Ga \subseteq F_i$ for all $i\le \stil$ and $\Ga \not\subseteq F_i$ for all $i> \stil$, where $\stil \le s$.  By Lemma \ref{Sym2a} we may further assume that
 $$\Gamma = \langle F_1\cap [r'], ..., F_{s'} \cap [r'] \rangle,$$
 where $s'\le \stil$. Assume that 
 $$[r] \setminus F_i = \{t_{i1},...,t_{i m_i} \} \cap [r'] =  \{ t_{i1},..., t_{i m'_i} \}.$$
 Similar to $Q_{\Gamma,n}$ (see (\ref{ECM11})) and $P_{\Gamma,n}$ (see (\ref{ECM71})), we consider the following polyhedron $Q'_{\Gamma,n}$ defined by
 \begin{equation} \label{ES22}
 Q'_{\Gamma,n} \ \ \begin{cases} x_{t_{11}} +  x_{t_{12}} + \cdots +  x_{t_{1m'_1}} \le n-1,\\
 \cdots  \\
 x_{t_{s'1}} +  x_{t_{s'2}} + \cdots +  x_{t_{s' m'_{s'}}} \le n-1,\\
 x_{t_{(s'+1) 1}} +  x_{t_{(s'+1)  2}} + \cdots +  x_{t_{(s'+1)  m_{(s'+1) }}} \ge n,\\
 \cdots  \\
 x_{t_{\stil 1}} +  x_{t_{\stil  2}} + \cdots +  x_{t_{\stil  m'_{\stil }}}  \ge n, \\
 x_j \ge 0 \ (j=1,...,r'),
 \end{cases}
 \end{equation}
 and the polyhedron   $P'_{\Gamma,n}$ defined by
 \begin{equation} \nonumber
P'_{\Gamma,n} \ \  \begin{cases} x_{t_{11}} +  x_{t_{12}} + \cdots +  x_{t_{1m'_1}} \le n, \\
 \cdots  \\
 x_{t_{s'1}} +  x_{t_{s'2}} + \cdots +  x_{t_{s' m'_{s'}}} \le n, \\
 x_{t_{(s'+1) 1}} +  x_{t_{(s'+1)  2}} + \cdots +  x_{t_{(s'+1)  m_{(s'+1) }}} \ge n, \\
 \cdots  \\
 x_{t_{\stil 1}} +  x_{t_{\stil  2}} + \cdots +  x_{t_{\stil  m'_{\stil }}}  \ge n,  \\
 x_j \ge 0 \ (j=1,...,r').
 \end{cases}
 \end{equation}
 Then, using Lemma \ref{Sym2a}, one can see that all similar results to those of Theorem \ref{CM7} to Theorem \ref{CM10}(i) hold for $I^{(n)}$, where one should replace $\Inba,\ P_{\Gamma,n}, Q_{\Gamma,n}$ by $I^{(n)},\ P'_{\Gamma,n}, \ Q'_{\Gamma,n}$, respectively, and the matrix $\At$ and the vectors $\widetilde{\bbf},\ \cT$ are now defined by (\ref{ES22}). In this new setting, all entries of $\At,\ \widetilde{\bbf},\ \cT$ are either $\pm 1$ or $0$. Therefore, by Hadamard's inequality,
 $D^2 \le r^r$, whence $D\le r^{r/2}$ and
$$2r^2D^3 \le 2r^{2+3r/2}.$$
Using  the Laplace expansion along the last column of a determinant of $[\At,\ \widetilde{\bbf}, \ \cT]$  and $[\At,\ \widetilde{\bbf}]$ we get
$\Delta \le (r+2)D'$ and $D' \le (r+1)D$. Set 
$$N_* = (r+1)D'(D-1) + (r+2)\Delta  - D .$$
 Then 
$$\begin{array}{ll}
N_* & <  (r+1)^2D(D-1) + (r+2)^2(r+1) D \\
& \le (r+1)^2 r^{r/2}(r^{r/2}-1) + (r+2)^2(r+1) r^{r/2}  .
\end{array}$$
If $r=2,3$ a direct computation shows that the last right hand side is less than $2r^{2+3r/2}$. Assume that $r\ge 4$. Then
$$\begin{array}{ll} N_* & < (r+1)^2r^r + (r+2)^2(r+1) r^{r/2} = (r+1)r^r[r+1 + \frac{(r+2)^2}{r^{r/2}}]\\
&\le (r+1)r^r[r+1 + \frac{(r+2)^2}{r^2} < (r+1)r^r (r+4) <  2r^{2+3r/2}. \end{array}$$
 By a similar statement to Theorem \ref{CM10}(i) we get that $a_i(R/I^{(n)})$ is a quasi-linear function of the same slope for all $n\ge  2r^{2+3r/2}$.
\end{proof}

In a recent preprint \cite{DHHT},    L. X. Dung et al.  were able to construct an example of square-free monomial ideal $I$ for which $\reg (I^{(n)})$ is not an eventually linear function of  $n$. We would like to conclude this section with
\vskip0.5cm
 \begin{quest}  Assume that $I$ is a square-free monomial ideal.  Is there a number $n_0 = O(r^k)$, where $k$ does not depend on $r$, such that $a_i(R/I^{(n)})$ and $\reg (I^{(n)})$ become 
quasi-linear functions for all $n\ge n_0$?
\end{quest}

The above question has an affirmative answer for a special class of square-free monomial ideals. Let 
$$I = \cap_{i=1}^s (X_{t_{i1}},...,X_{t_{im_i}}),$$
and
$$a_{ij} = \begin{cases} 1 & \text{if} \ j = t_{il} \ (l=1,...,m_i),\\
0& \text{otherwise}.
\end{cases}$$
If the $s\times r$-matrix $A= (a_{ij})$ is totally unimodular, i.e., all subdeterminants of $M$ are either $\pm 1$ or $0$, then the main results of \cite{HaT} state that $a_i(R/I^{(n)})$  and $\reg (I^{(n)})$ are linear functions of $n$ for all $n\ge r^2$. In this situation, all polyhedra $P'_{\Gamma ,n}$ and $Q'_{\Gamma ,n}$ are integral, that means their vertices are integer points, so that the optimal value of the  integer program  with the set of feasible solutions $P'_{\Gamma ,n} \cap \Nset^{r'}$ is already a linear function of $n$. The main efforts in \cite{HaT} are devoted to the corresponding integer program  with the set of feasible solutions  $Q'_{\Gamma ,n} \cap \Nset^{r'}$. Note that  applying Corollary \ref{O3I} and Theorem \ref{CM10} we only can conclude that all $a_i(R/I^{(n)})$, and hence also $\reg (I^{(n)})$, are linear functions of $n$ for all $n\ge (r+1)(r+2)^2$. 

 \subsection*{Acknowledgment} The author would like to thank professors Martin Gr\"otschel, Jesus De Loera, Alexander Barvinok and Dr. Hoang Nam Dung for their consultation on Linear and Integer Programming. In particular, thanks to them I am aware of references \cite{Sh, StT,Wol, Wo}, and the proof of Proposition \ref{O3} was simplified.
 
 The author also would like to thank both referees for their critical comments and useful suggestions, especially for pointing out the reference \cite{Go}.

 This work is partially supported by NAFOSTED (Vietnam) under the grant number 101.04-2018.307.

\end{document}